 \documentclass[draft]{article}

\usepackage{amsmath,amsfonts,amsthm,amssymb,amscd,cancel,color,hhline}
\usepackage{enumitem}
\usepackage{verbatim}
\usepackage[dvipdfmx]{graphicx}
\usepackage{ulem,color}

\renewcommand{\Re}{\mathop{\rm Re}\nolimits}
\renewcommand{\Im}{\mathop{\rm Im}\nolimits}

\theoremstyle{plain}
\newtheorem{theorem}{Theorem}[section]
\newtheorem{lemma}[theorem]{Lemma}
\newtheorem{proposition}[theorem]{Proposition}

\theoremstyle{definition}

\theoremstyle{remark}
\newtheorem{remark}[theorem]{Remark}

\newtheorem{claim}[theorem]{Claim}

\newcommand{\R}{{\mathbb R}}

\newcommand{\Z}{{\mathbb Z}}

\newcommand{\N}{{\mathbb N}}

\def\im{{\rm i}}

\newcommand{\C}{\mathbb{C}}

\def\({\left(}
\def\){\right)}
\def\[{\left[}
\def\]{\right]}
\def\<{\left\langle}
\def\>{\right\rangle}

\numberwithin{equation}{section}

\begin{document}

\title{Revisiting asymptotic stability of solitons of nonlinear Schr\"odinger equations via refined profile method}

\author{Scipio Cuccagna, Masaya Maeda}
\maketitle

\begin{abstract}
In this paper, we give an alternative proof for the asymptotic stability of solitons for nonlinear Schr\"odinger equations with internal modes.
The novel idea is to use ``refined profiles" developed by the authors for the analysis of small bound states.
By this new strategy, we able to avoid the normal forms.
Further, we can track the functions appearing in the Fermi Golden Rule hypothesis.
\end{abstract}

\section{Introduction}

In this paper  we revisit a theorem on the asymptotic stability of ground states of the nonlinear Schr\"odinger equations (NLS), see \cite{Cuccagna11CMP,Cuccagnatrans,Bambusi13CMPas},
giving a novel and much simplified proof, thanks to notion of "Refined Profile", which allows to avoid the normal form arguments in the older papers.

To set up the problem,
 we consider   the scalar NLS,
 \begin{align}\label{scalarNLS}
 \im \partial_t u = -\Delta u + g(|u|^2)u \text{  with } u(t,x):\R^{1+3} \to {\C},
 \end{align}
 where $g \in C^\infty (\R,\R)$ with $g(0)=0$ satisfies the growth condition:
 \begin{align}\label{eq:ggrowth}
 \ \forall n=0,\cdots,4,\ \exists C_n>0,\  |g^{(n)}(s)|\leq  C_n s^{2-n}.
 \end{align}
 NLS \eqref{scalarNLS} under these conditions is locally well-posed in $H^1(\R^3, \C )$, see e.g.\ Theorem 5.5 of \cite{LPBook}.

 We will assume the existence of ground states.
 In particular, we  assume    existence of an open interval $\mathcal{O}\subset (0,\infty)$ and of a  map
   \begin{align}\label{eq:mapgs}
 \  \omega \mapsto \varphi_\omega  \in C^\infty (\mathcal{O}, H^1_{\mathrm{rad}}\cap L^\infty(\R^3, \C ) )  \text{ , }    H^1_{\mathrm{rad}}(\R^3, \C ):=\{u\in H^1(\R^3, \C )\ |\ u (x) \equiv u(|x|) \} ,
 \end{align}
 where $\varphi_{\omega}$    a ground state, i.e.\  it satisfies
 \begin{align}\label{eq:sp}
 -\Delta \varphi_\omega + \omega \varphi_\omega + g(\varphi_\omega ^2) \varphi_\omega=0   \text{   and $\varphi_\omega(x)>0$ for all $x\in \R^3$}.
 \end{align}
 For a very general existence result, see \cite{BL83ARMA}.

 We fix $\omega_*\in \mathcal{O}$ and assume the following two hypotheses:
 \begin{itemize}
 	\item[(H1)] $L_{\omega_*,+}$ has exactly one negative eigenvalue and $\ker \left.L_{\omega_*,+}\right|_{L^2_{\mathrm{rad}}}=\{0\}$,   where $$L_{\omega,+}:=-\Delta + \omega + g(\varphi_{\omega}^2)+2g'(\varphi_\omega^2)\varphi_\omega^2;$$
 	\item[(H2)] $ \left.\frac{d}{d\omega}\right|_{\omega=\omega_*}  \| \varphi_\omega \| _{L^2 (\R ^3)} ^2   >0$.
 \end{itemize}

\begin{remark}
By (H1), we have $\mathrm{ker}L_{\omega_*,+}=\{\partial_{x_l}\varphi\ |\ l=1,2,3\}$, see \cite{Weinstein85, CGNT07SIMA}.
\end{remark}

 \begin{remark}
 Both conditions (H1) and (H2) hold  for $\omega$ near $\omega_*$.
  In the following, we will restrict $\mathcal{O}$ so that for all $\omega\in \mathcal{O}$, assumption (H1) and (H2) hold.
  \end{remark}

 The second condition in (H1) is the so-called nondegeneracy condition, for sufficient conditions that insure it, see \cite{ASW18AMS,LN20CVPDE} and reference therein.
 The first condition in (H1) holds when $\varphi_{ \omega}$ is obtained by variational arguments, see Proposition B.1 of \cite{FGJS04CMP}.

 The condition (H1) and the Vakhitov-Kolokolov condition (H2) are standard sufficient conditions to ensure the orbital stability of $e^{\im \omega t} \varphi_\omega$ for $\omega=\omega_*$.

 \begin{proposition}[Orbital stability]\label{prop:os}
 There exist $\epsilon_0>0$ and $C>0$ s.t.\ if $
 \|u(0)-\varphi_\omega\|_{H^1}<\epsilon_0$, then
 \begin{align}
 \sup_{t\in\R}\inf_{\theta\in \R,y\in \R^3}\|u(t)-e^{\im \theta}\varphi_\omega(\cdot-y)\|_{H^1}\leq C \|u(0)-\varphi_\omega\|_{H^1},
 \end{align}
 where $u(t)$ is the solution of \eqref{scalarNLS}.
 \end{proposition}

 \begin{proof}
 See Theorem 3.4 of  \cite{GSS}.
 \end{proof}

\begin{remark}
The number of negative eigenvalues of $L_{\omega_*,+}$ is called Morse index.
 In Remark 1   \cite{Bambusi13CMPas}, it  is stated that Orbital Stability follows from (H1) and (H2) without assuming   Morse index    $1$, by quoting \cite{FGJS04CMP} which claims  that   Morse index   $1$   follows from the second condition of (H1), (H2) and  Theorem 3   \cite{GSS}.
 However, Theorem 3    \cite{GSS}  is proved  assuming that Morse index is $1$. So for Orbital Stability it seems that we need  the hypotheses as we state them here.
 We are not aware of any example of positive bound states with Morse index more than $2$, satisfying the second  condition in (H1) and (H2).
 However, for NLS with potential and  for  systems of NLS there are such examples, see \cite{M10}.
 Also, there exists a positive bound state with Morse index $2$ (although it is not clear if this bound state satisfies (H2)), see \cite{DdPG13PLMS}.
 \end{remark}

   The aim of this paper is to prove a stronger stability property,  the asymptotic stability, which states that all solutions near the ground state $\varphi_{ \omega}$ converge to $\varphi_{ \omega_+}$ for some $\omega_+$ near $\omega_*$ modulo scattering waves.
   Postponing assumptions (H3)--(H7), the main theorem of this paper is as follows, already known under slightly stronger assumption:

      \begin{theorem}\label{thm:main}
      	Assume (H1)--(H7) hold, where (H3)--(H7) are given below.
      	Then, there exist $\epsilon _0>0$ and $C>0$ s.t.\ for all $u_0\in H^1$ satisfying $\|u_0-\varphi_{\omega_*}\|_{H^1} <\epsilon _0$, there exist $C^1(\R )$ functions $\theta,\omega,y,v$ and  there are  $\eta_+\in H^1(\R^ 3)$, $v_+\in \R^3$ and $\omega_+\in \mathcal O$ s.t.\
      	\begin{align}
      	\lim_{t\to \infty}\|u(t)-e^{\im \theta(t)}e^{ \frac{\im}{2}v(t)\cdot x}\varphi_{\omega(t)}(\cdot-y(t)) - e^{\im t \Delta}\eta_+ \|_{H^1}=0,\label{thm:main1} \\
      	\lim_{t\to \infty}|\omega(t)-\omega_+|=\lim_{t\to \infty}|v(t)-v_+|=0,\label{thm:main2}
      	\end{align}
      where $u(t)$ is the solution of \eqref{scalarNLS} satisfying $u(0)=u_0$, and
      	\begin{align}
      	\|\eta_+\|_{H^1}+|v_+|+|\omega_+-\omega_*|\leq C\|u_0-\varphi_{\omega_*}\|_{H^1}.\label{thm:main3}
      	\end{align}
      \end{theorem}

%
      The key novelty here is the fact that we avoid the normal forms in the context of the Fermi Golden Rule  (FGR).
      This is a significant advance because
      the FGR is a key  mechanism
      in  radiation induced dissipation. Classical oscillating mechanisms, like the oscillations of a soliton trapped by
      a potential, which in certain asymptotic regimes are known to last for long times, see for example \cite{JLFGS06AHP}, are
      not expected to last for ever. Similar oscillatory motions, in correspondence to critical points of the function $\omega \to \| \varphi _\omega \| _{L^2}$, which appear naturally in the case a saturated versions of the $L^2$ critical pure power focusing  NLS, see \cite{BG01MCS,MRS10JNS}, analyzed rigorously in
      \cite{CM20JDE}, are not expected to hold for ever.  Other related examples of complicated oscillatory patterns, lasting over long times
      are the complicated patterns near  branchings of the maps $\omega \to \varphi _\omega$ considered  in \cite{MW10DCDS}, which again are
      expected to be transient.      In analogy to  the  role of the FGR in the stabilization phenomena observed in \cite{BP2}, \cite{SW3}--\cite{TY1}, \cite{Cuccagna11CMP,Cuccagnatrans} and many other papers, some of whom referenced in the survey \cite{CM21DCDS},
      what breaks the oscillations should be an exchange of energy between discrete and continuous modes of the solutions. In particular,
      in each of   \cite{BG01MCS,JLFGS06AHP,MRS10JNS,MW10DCDS} the linearization   $\mathcal{H}_{\omega }$  has a pair of  eigenvalues very close to the origin.  The nonlinear interaction of the corresponding discrete modes with the continuous modes, should be responsible for transient nature of the patterns observed.  The longevity of these patterns is   connected with the  smallness of the eigenvalues of the pair, because
      the nonlinear interaction, which leads to radiation induced dissipation on discrete modes,  is related to the fact that multiples of the eigenvalues are in the continuous spectrum, see the definition of
      resonant multi-indexes under \eqref{def:prec}. In the present paper we avoid the issue of small eigenvalues, see that in (H5) we are assuming $\lambda_j(\omega)>0$ and in particular $\min_{j}\lambda_j(\omega_*)>0$,
      but we expect that the main novel idea of this paper, and of  the previous papers \cite{CM21AnnPDE,CM21}, might have some relevance also in the case
      of small eigenvalues.  This because, in the presence of small eigenvalues,  the problem of simplifying as
      much as possible the search for optimal coordinate systems, where it might be easier to see  the    radiation induced dissipation,
      becomes essential, in view of the large number of steps required in normal forms arguments. Now it turns out that with a well chosen
      "Refined Profile", the resulting coordinates are automatically optimal. This is similar to, and in fact was inspired by,   what happens in the study of the $\log \log$  blow up in the $L^2$ critical NLS,
      see  \cite{perelman01,MR03GAFA,MR4}, where the choice of an   appropriate deformation of the ground states, yields automatically
      to a system
      where the dissipation mechanism is directly available.   The advantage of the Refined Profile, that is of an appropriate deformation of the
      ground states  which incorporates all the discrete coordinates, is that here as well as in  \cite{CM21AnnPDE,CM21}, is that it
      can be defined by an elementary argument. Obviously, the method will have to be tested  to study the transient nature
      of the patterns in \cite{JLFGS06AHP,MRS10JNS,MW10DCDS,CM20JDE}, and in other analogous contexts, and in general to get truly novel results.

   \subsection{Linearized operator and assumption (H3)--(H6)}\label{subsec:Assumptions}
  We use the Pauli matrices
   \begin{equation}
   \label{pauli-matrices}
   \sigma_1=\begin{pmatrix} 0 &
   1 \\
   1 & 0
    \end{pmatrix} \,,
   \quad
   \sigma_2= \begin{pmatrix} 0 &
   -\im \\
    \im  & 0
    \end{pmatrix} \,,
   \quad
   \sigma_3=\begin{pmatrix} 1 &
   0 \\
   0 & -1
    \end{pmatrix}.
   \end{equation}
   For given function $\psi$, we define
   \begin{align}\label{eq:lineariz_gen}
           \mathcal{H}[\omega,\psi]:&=\begin{pmatrix}
         -\Delta + \omega + g(|\psi|^2)+g'(|\psi|^2)|\psi|^2 & g'(|\psi|^2)\psi^2\\
         -g'(|\psi|^2)\overline{\psi^2} &  \Delta - \omega - g(|\psi|^2)- g'(|\psi|^2)|\psi|^2
         \end{pmatrix},
         \end{align}
         and for  $\omega\in \mathcal{O}$, we consider the "linearized operator"
      \begin{align}\label{eq:lineariz}
        \mathcal{H}_\omega:=\mathcal{H}[\omega,\varphi_{ \omega}].
      \end{align}
      \begin{remark}
      Setting $u=e^{\im \omega t}(\varphi_\omega +r)$ and substituting this into \eqref{scalarNLS}, we obtain $$\im \partial_t r = -\Delta r + \omega r + g(\varphi_\omega^2)r + g'(\varphi_{ \omega}^2)\varphi_{ \omega}^2r +g'(\varphi_{ \omega}^2)\varphi_{ \omega}^2\overline{r} +O(r^2).$$
      Since   complex conjugation is not $\C$-linear, it is natural to consider the above matrix form of the linearized operator when considering the spectrum.
      \end{remark}

Under the assumptions  (H1) and (H2), the generalized kernel $ {N}_g(\mathcal{H}_\omega):=\cup_{j=1}^\infty \mathrm{ker} \mathcal{H}_{\omega}^j$ becomes
\begin{align}\label{eq:Ng}
 {N}_g(\mathcal{H}_{\omega })=\mathrm{span}\{ \im\sigma _3 \phi_{\omega }, \partial_\omega \phi_{\omega }, \partial_{x_l}\phi_{\omega }, \im \sigma _3 x_l\phi_{\omega },\ l=1,2,3 \},\ \text{where}\ \phi_\omega=\begin{pmatrix}
 \varphi_{ \omega} \\ \varphi_{ \omega}
 \end{pmatrix}.
\end{align}
\begin{remark}
The inclusion  $\supseteq$ always holds while $\subseteq$  follows from (H1) and (H2), see  \cite{Weinstein85}.
\end{remark}

Under the assumption of (H1) and (H2) (and the fact that $\varphi_\omega$ is positive), one can show $\sigma(\mathcal{H}_{\omega }) \subset \R$ (otherwise the bound state will be unstable contradicting Proposition \ref{prop:os}) and $\sigma_{\mathrm{ess}}(\mathcal{H}_{\omega })=(-\infty,-\omega ] \cup [\omega ,\infty)$, where $\sigma(\mathcal{H}_{\omega }) $ and $\sigma_{\mathrm{ess}}(\mathcal{H}_{\omega })$ are the  spectrum and essential spectrum respectively.
We assume:
\begin{itemize}
	\item [(H3)] $\pm \omega_*$ are not eigenvalues nor resonance of $\mathcal{H}_{\omega_*}$;
	\item[(H4)] $\mathcal{H}_{\omega_*}$ has  no   eigenvalues   in  $(-\infty,-\omega_* ) \cup (\omega _*,\infty)$  (no embedded eigenvalues).
\end{itemize}

\begin{remark}
Assumption (H3)  is   generically true, while we expect assumption (H4)  always to be true.
That is, we conjecture the absence of embedded eigenvalues
with positive Krein signature.
Notice that the Krein signature of such embedded eigenvalues has to be positive when $\varphi_\omega$
is a ground state.
\end{remark}

The spectrum of $\mathcal{H}_\omega$ is symmetric with respect to the imaginary axis.
It is known that there are finitely many eigenvalues with finite total multiplicity, Proposition 2.2 of \cite{CPV}.
Thus, considering the Riesz projection, we see that the projections to the finite dimensional subspaces of discrete components are smooth in $\omega$.
We assume the following:
\begin{itemize}
\item [(H5)]
There exist $N\in \N_0$ and $\lambda_j(\cdot)\in C^\infty(\mathcal{O},\R_+)$ and $\xi_j[\cdot]\in C^\infty(\mathcal{O},L^2(\R))$ for $j=1,\cdots,N$ s.t. $\sigma_d(\mathcal{H}_\omega)=\{0\}\cup\{\pm\lambda_j(\omega),\ j=1,\cdots,N\}$ and $\mathcal{H}_\omega \xi_{j}[\omega]=\lambda_j(\omega)\xi_j[\omega]$.
\end{itemize}
\begin{remark}
Assumption (H5) is satisfied   when $g$ is analytic.
\end{remark}

We write
\begin{align}\label{def:xipm}
\xi_j[\omega]=\begin{pmatrix}
\xi_{j+}[\omega] \\ \xi_{j-}[\omega]
\end{pmatrix} .
\end{align}
From the anticommutative relation $\sigma_1 \mathcal{H}_{\omega}=-\mathcal{H}_{\omega}\sigma_1$, one can see $\sigma_1 \xi_j[\omega]$ is the eigenvector of the eigenvalue $-\lambda_j(\omega)$.
It is possible to take all the  $\xi_{j\pm }[\omega]$  to be $\R$-valued and moreover  normalize so that, for $\delta_{jk}$ is the Kronecker's delta,
\begin{align}\label{eq:krein}
 ( \sigma _3 \xi_j[\omega] , \xi_j[\omega] ) =  \delta _{jk}.
\end{align}
\begin{remark}
The above equality is always true for $j\neq k$, while for $j=k$ it reflects the nontrivial, but easy to prove,  fact that  each eigenvalue
$ \lambda_j(\omega)$  has positive Krein signature (this is a consequence of the fact that  $\varphi_\omega$
is a ground state).
\end{remark}


%

\noindent By standard argument, we know that $\varphi_\omega$ and $\xi_j[\omega]$ decay  exponentially, see \cite{HL07BullLMS}.
Thus, we can show that for all $\omega\in \mathcal{O}$ we have
$\phi_{\omega }, \xi_j[\omega] \in \Sigma$, where for sufficiently large $\sigma>0$, $\Sigma$ is defined by
\begin{align}\label{eq:defsig}
\Sigma:=\{u\in L^2(\R ^3,\C ^2)\ |\ \|u\|_{\Sigma}<\infty\},\ \|u\|_{\Sigma}:=\|    \< x \> ^\sigma  u\|_{H^2}.
\end{align}
The map $\omega\mapsto \xi_{j }[\omega]$ is $C^\infty$ in $\Sigma$ and the same holds for $\varphi_\omega$ too.

In the following, given $\mathbf{x}\in \mathbb{K}^M$ for $\mathbb{K}=\N_0,\R,\C$ and $M\in \N$ with $\mathbf{x}=(x_1,\cdots,x_M)$,
 we set $\|\mathbf{x}\|:=\sum_{n=1}^{M}|x_n|$.
To state further assumptions on the discrete spectrum, we introduce further notation.
For $\mathbf{m}\in \N_0^{2N}$, we write $\mathbf{m}=(\mathbf{m}_+,\mathbf{m}_-)$, where $\mathbf{m}_\pm \in \N_0^N$.
We also set $\overline{\mathbf{m}}=(\mathbf{m}_-,\mathbf{m}_+)$, $\mathbf{e}^{j}=(\delta_{1j},\cdots,\delta_{Nj}) \in \N_0^N$, $\mathbf{e}^{j+}=(\mathbf{e}^j,0)$, $\mathbf{e}^{j-}=\overline{\mathbf{e}^{j+}}$
and
	\begin{align}
	\lambda(\omega,\mathbf{m})=\sum_{j=1}^N\lambda_j(\omega) \(m_{+,j} -m_{-,j}\).
	\end{align}
For $\mathbf{m},\mathbf{m}'\in \N_0^{2N}$, we define
\begin{align}
\mathbf{m}'\preceq \mathbf{m}\ &\Leftrightarrow\ m_{+,j}'+m_{-,j}'\leq m_{+,j}+m_{-,j},\ \text{for\ all\ }j=1,\cdots,N,\nonumber\\
\mathbf{m}'\prec \mathbf{m}\ &\Leftrightarrow\ \mathbf{m}'\preceq \mathbf{m}\ \text{and}\ \|\mathbf{m}'\|<\|\mathbf{m}\|.\label{def:prec}
\end{align}
We define the resonant resp. minimal resonant indices as
\begin{align*}
\mathbf{R}_\omega=\{\mathbf{m}\in \N_0^{2N}\ |\ |\lambda(\omega,\mathbf{m})|>\omega\}\text{   resp. }\mathbf{R}_{\mathrm{min},\omega}=\{\mathbf{m}\in \mathbf{R}_\omega\ |\ \not\exists\mathbf{m}'\in \mathbf{R}_\omega\ s.t.\ \mathbf{m}'\prec \mathbf{m}\}.
\end{align*}
Further,   the set of indices which we will ignore is
\begin{align*}
\mathbf{I}_\omega:=\{\mathbf{m}\in \N_0^{2N}\ |\ \exists\mathbf{m}'\in \mathbf{R}_{\mathrm{min},\omega}\ s.t.\ \mathbf{m}'\prec \mathbf{m}\}.
\end{align*}

We assume the following, on the discrete spectrum.
\begin{itemize}
	\item [(H6)]
	We assume that for $\mathbf{m}_+\in \N_0^N$ with $\|\mathbf{m}_+\|\geq 2$, $\lambda(\omega_*,(\mathbf{m}_+,0))\neq \lambda_j(\omega_*)$ for $j=1,\cdots, N$ and
	\begin{align*}
	\forall \mathbf{m}\in \N_0^{2N}\setminus\mathbf{I}_{\omega_*},\ |\lambda(\omega_* ,\mathbf{m})|\neq \omega_*.
	\end{align*}
\end{itemize}

\begin{remark}
A sufficient condition for (H6) is that 
\begin{align*}
2\leq \|  \mathbf{m}_+\|\leq \omega_*\(\min_{j}\lambda_j(\omega_*)\)^{-1}\ \Rightarrow\
 \lambda(\omega_*,(\mathbf{m}_+,0))\neq  \lambda_1(\omega_*),\cdots,\lambda_N(\omega_*),\ \omega_*.
\end{align*}
\end{remark}

Under (H6), restricting $\mathcal{O}$ if necessary, $\mathbf{R}_{\mathrm{min},\omega}$ and $\mathbf{I}_\omega$ do not depend on $\omega\in \mathcal{O}$.
Thus, we write them $\mathbf{R}_{\mathrm{min}}$ and $\mathbf{I}$ respectively.
We enumerate the set $\{\lambda(\omega_*,\mathbf{m}) |\ \mathbf{m}\in \mathbf{R}_{\mathrm{min}}\}$ as $\{\pm r_k\ |\ k=1,\cdots,M\}$ where $r_k>0$ and set
\begin{align}\label{def:Rmink}
\mathbf{R}_{\mathrm{min},k}=\{\mathbf{m}\in \mathbf{R}_{\mathrm{min}}\ |\ \lambda(\omega_*,\mathbf{m})=r_k\},
\end{align}
and write $\mathbf{R}_{\mathrm{min},k}=\{\mathbf{m}(k,n)\ |\ n=1,\cdots,M_k\}$.
The set of nonresonant indices defined by
\begin{align*}
\mathbf{NR}:= \N_0^{2N}\setminus (\mathbf{R}_{\mathrm{min}}\cup\mathbf{I}).
\end{align*}
Notice that we have
$
\N_0^{2N}\setminus\mathbf{I}=\mathbf{R}_{\mathrm{min}}\cup \mathbf{NR}.
$
We further set
\begin{align*}
\Lambda_{0}:=\{\mathbf{m}\in \mathbf{NR}\setminus\{0\}\ |\ \lambda(\omega_*, \mathbf{m})=0\}\text{   and }\Lambda_{j}:=\{\mathbf{m}\in \mathbf{NR}\ |\ \lambda(\omega_*, \mathbf{m})=\lambda_j(\omega_*)\}.
\end{align*}
Finally, for $\mathbf{z}=(z_1,\cdots,z_N)\in \C^N$, we write
\begin{align*}
\mathbf{z}^{\mathbf{m}}=\mathbf{z}^{\mathbf{m}_+}\overline{\mathbf{z}}^{\mathbf{m}_-},\ \mathbf{m}=(\mathbf{m}_+,\mathbf{m}_-)\in \N_0^{2N},\ \text{where}\ \mathbf{z}^{\mathbf{m}_\pm}=\prod_{j=1}^N z_{j}^{m_{j\pm}}.
\end{align*}

\subsection{Refined profile and Fermi Golden Rule assumption (H7)}

 For a $C^1$ function in $\mathbf{z}$ let  $DF\mathbf{w}=D_{\mathbf{z}}F(\mathbf{z})\mathbf{w}=\left.\frac{d}{d\epsilon}\right|_{\epsilon=0}F(\mathbf{z}+\epsilon\mathbf{w})$.
 Let also  $\nabla_x=(\partial_{x_1},\partial_{x_2},\partial_3)$.

We now introduce the notion of refined profile.


\begin{proposition}\label{prop:rp_pre_galilei}
There exist  $\varphi[\omega,\mathbf{z}]$, $\widetilde{\theta}(\omega,\mathbf{z})$, $\widetilde{\omega}(\omega,\mathbf{z})$, $\widetilde{\mathbf{y}}(\omega,\mathbf{z})$, $\widetilde{\mathbf{v}}(\omega,\mathbf{z})$ and $\widetilde{\mathbf{z}}(\omega,\mathbf{z})$ smoothly defined in the neighborhood of $(\omega_*,0)\in \mathcal{O}\times \C^N$, such that $\varphi[\omega,0]=\varphi_{ \omega}$ and for $\varphi = \varphi[\omega,\mathbf{z}]$,
\begin{align}\label{eq:phi_pre_gali}
\mathcal{R}[\omega,\mathbf{z}]:=-\Delta \varphi + g(|\varphi|^2)\varphi + \widetilde{\theta}\varphi - \im \widetilde{\omega}\partial_{\omega}\varphi + \im \widetilde{\mathbf{y}} \cdot \nabla_x \varphi+\frac{1}{2}\widetilde{\mathbf{ v}}\cdot x \varphi -\im D_{\mathbf{z}}\varphi \widetilde{\mathbf{z}},
\end{align}
can be expanded as
\begin{align}&
\mathcal{R}[\omega,\mathbf{z}]=\sum_{\mathbf{m} \in \mathbf{R}_{\mathrm{min}}}\mathbf{z}^{\mathbf{m}}G_{\mathbf{m}}+\mathcal{R}_1[\omega,\mathbf{z}] \text{   with }G_{\mathbf{m}}\in \Sigma \text{   and}  \label{expestR}
\\& \label{estR}
\| \mathcal{R}_1\|_{\Sigma}\lesssim  \(|\omega-\omega_*|+\|\mathbf{z}\|\)\sum_{\mathbf{m}\in \mathbf{R}_{\mathrm{min}}}|\mathbf{z}^{\mathbf{m}}|.
\end{align}
Furthermore, $\mathcal{R}[\omega,\mathbf{z}]$ satisfies the following orthogonality conditions, for $\< f,g\> :=\Re \int f \overline{g }dx$,
\begin{align}\label{R:orth} &
\<\mathcal{R}[\omega,\mathbf{z}],\im\varphi [\omega,\mathbf{z}]\>=\<\mathcal{R}[\omega,\mathbf{z}], \partial_{\omega}\varphi [\omega,\mathbf{z}]\>=\<\mathcal{R}[\omega,\mathbf{z}],  \partial_{x_l}\varphi[\omega,\mathbf{z}]\>  =  \<\mathcal{R}[\omega,\mathbf{z}],\im x_l\varphi[\omega,\mathbf{z}]\> \\& =\<\mathcal{R}[\omega,\mathbf{z}],\partial_{z_{jA}}\varphi[\omega,\mathbf{z}]\> \equiv 0,\text{   for all $l=1,2,3$, $j=1,\cdots,N$, and $A=R,I$,}\nonumber
\end{align}
where $z_{jR}=\Re z_j$  and    $z_{jI}=\Im z_j$.

\end{proposition}

We set
\begin{align}\label{FermiG}
\mathfrak{G}_{\mathbf{m}}=\begin{pmatrix}
G_{\mathbf{m}}\\
G_{\overline{\mathbf{m}}}
\end{pmatrix},
\end{align}
and the wave operator $W$ by
\begin{align}\label{def:waveop}
W=\lim_{t\to \infty} e^{\im t \mathcal{H}_{\omega_*}}e^{-\im t \sigma_3(-\Delta+\omega_*)}.
\end{align}
For the existence and boundedness, as well as its adjoint $W^*$ and inverse see \cite{Cu1,CPV}.

We state now our final assumption, the  Fermi Golden Rule (FGR).

\begin{itemize}
	\item[(H7)] For each $k=1,\cdots,M$, $\mathcal{F}{\(W^*\mathfrak{G}_{\mathbf{m}(k,1)}\)_+},\cdots,\mathcal{F}{\(W^*\mathfrak{G}_{\mathbf{m}(k,M_k)}\)_+}$ are linearly independent as a function on the sphere $|\xi|^2=r_k-\omega_*$.
Here, $\mathcal{F}{f}$ is the Fourier transform of $f$ and $(F)_+$ is the upper component of the $\C^2$-valued function $F$.
\end{itemize}

\begin{remark}
If $M_k=1$, (H7) states that there exists some $\xi$ with $|\xi|^2=r_k$ such that  $\widehat{\mathfrak{G}}_{\mathbf{m}}(\xi)\neq 0$.
In generic situations, when is the case $\lambda_j(\omega_*)$ are $\Z $--linearly independent and each eigenspace of $\mathcal{H}_{\omega}$ is spanned by a finite subgroup of rotations of one element, (H7) is generic.
The $ G_{\mathbf{m}}$,
   are  obtained by an elementary recursive linear procedure, much simpler than the analogous one in \cite{Cuccagna11CMP,Cuccagnatrans,Bambusi13CMPas},
 which  involves various nonlinear  normal forms transformations.   The theory in this paper should make much more feasible
the task of checking numerically the FGR hypothesis for specific examples.
\end{remark}

\section{Proof of Proposition \ref{prop:rp_pre_galilei}}\label{sec:rf}

In this section, we provide the proof of Proposition \ref{prop:rp_pre_galilei}.
\begin{proof}[Proof of Proposition \ref{prop:rp_pre_galilei}]
We seek for $\varphi$, $\widetilde{\theta}$, $\widetilde{\omega}$, $\widetilde{\mathbf{y}}$,
$\widetilde{\mathbf{v}}$ and $\widetilde{\mathbf{z}}$ having the following expansions:
\begin{align}\label{eq:anzphi}
\varphi[\omega,\mathbf{z}]&=\sum_{\mathbf{z}\in \mathbf{NR}}\mathbf{z}^{\mathbf{m}}\varphi_{\mathbf{m}}[\omega],\text{where}\
 \varphi_0[\omega]=\varphi_{ \omega},\ \text{and}\  \varphi_{\mathbf{e}_{j\pm}}[\omega]=\xi_{j\pm}[\omega]\ \text{for}\ j=1,\cdots,N,
\end{align}
and
\begin{align}
\widetilde{\theta}(\omega,\mathbf{z})&=\omega+ \sum_{\mathbf{m}\in \Lambda_0}\mathbf{z}^{\mathbf{m}}\widetilde{\theta}_{\mathbf{m}}(\omega)+\widetilde{\theta}_{\mathcal{R}}(\omega,\mathbf{z}),\label{eq:thetaanz},\
\widetilde{\omega}(\omega,\mathbf{z})=\widetilde{\omega}_{\mathcal{R}}(\omega,\mathbf{z}),\\
\widetilde{y}_l(\omega,\mathbf{z})&=\widetilde{y}_{l\mathcal{R}}(\omega,\mathbf{z}),\ \widetilde{v_l}(\omega,\mathbf{z})=\sum_{\mathbf{m}\in \Lambda_0}\mathbf{z}^{\mathbf{m}}\widetilde{v}_{l\mathbf{m}}(\omega)+\widetilde{v}_{l\mathcal{R}}(\omega,\mathbf{z}),\ l=1,2,3,\label{eq:vanz}\\
\widetilde{z_j}(\omega,\mathbf{z})&=-\im \lambda_j z_j-\im \sum_{\mathbf{m}\in \Lambda_j,\ \|\mathbf{m}\|\geq 2}\mathbf{z}^{\mathbf{m}}\widetilde{\lambda}_{j\mathbf{m}}(\omega)+\widetilde{z}_{j\mathcal{R}}(\omega,\mathbf{z}),\ j=1,\cdots,N,\label{eq:zanz}
\end{align}
with $\widetilde{\lambda}_{j \mathbf{e}_{j+}}(\omega)= \lambda_j(\omega)$ and
\begin{align}
|\widetilde{\theta}_{\mathcal{R}}|+|\widetilde{\omega}_{\mathcal{R}}|+\|\widetilde{\mathbf{y}}_{\mathcal{R}}\|+\|\widetilde{\mathbf{v}}_{\mathcal{R}}\| + \|\widetilde{\mathbf{z}}_{\mathcal{R}}\| \lesssim \sum_{\mathbf{m} \in \mathbf{R}_{\mathrm{min}}} |\mathbf{z}^{\mathbf{m}}|.\label{eq:Ranz}
\end{align}
Our task is to determine $\varphi_{\mathbf{m}}$, $\widetilde{\theta}_{\mathbf{m}}$, $\widetilde{\theta}_{\mathcal{R}}$, $\widetilde{\omega}_{\mathcal{R}}$, $\widetilde{\mathbf{y}}_{\mathcal{R}}$, $\widetilde{\mathbf{v}}_{\mathbf{m}}$, $\widetilde{\mathbf{v}}_{\mathcal{R}}$, $\widetilde{\mathbf{z}}_\mathbf{m}$ and $\widetilde{\mathbf{z}}_{\mathcal{R}}$ so that $\mathcal{R}$ given by \eqref{eq:phi_pre_gali} satisfies \eqref{expestR}--\eqref{R:orth}.

The proof consists of two  steps.
In the 1st, we substitute \eqref{eq:anzphi}--\eqref{eq:zanz} into the r.h.s.\ of \eqref{eq:phi_pre_gali} and solve the equation for each coefficients of $\mathbf{z}^{\mathbf{m}}$ for $\mathbf{m}\in \mathbf{NR}$.
This determines $\varphi_{\mathbf{m}}$, $\widetilde{\theta}_{\mathbf{m}}$,  $\widetilde{\mathbf{v}}_{\mathbf{m}}$, and $\widetilde{\mathbf{z}}_\mathbf{m}$.
Furthermore, since we have erased all coefficients of $\mathbf{z}^{\mathbf{m}}$ with $\mathbf{m}\in \mathbf{NR}$, the r.h.s.\ of \eqref{eq:phi_pre_gali}, which we will denote $\widetilde{\mathcal{R}}$ (see \eqref{def:tildeR} below), will satisfy the error estimate \eqref{estR} after subtracting the $\mathbf{z}^{\mathbf{m}}$ terms with $\mathbf{m}\in \mathbf{R}_{\mathrm{min}}$.
Next, in the 2nd step, we choose $\widetilde{\theta}_{\mathcal{R}}$, $\widetilde{\omega}_{\mathcal{R}}$, $\widetilde{\mathbf{y}}_{\mathcal{R}}$, $\widetilde{\mathbf{v}}_{\mathcal{R}}$  and $\widetilde{\mathbf{z}}_{\mathcal{R}}$ so that \eqref{R:orth} is satisfied.
In the 2nd step we are basically taking a projection of $\widetilde{\mathcal{R}}$ to satisfy the orthogonality conditions \eqref{R:orth}.

\subsubsection*{1st step}
We substitute \eqref{eq:anzphi}, \eqref{eq:thetaanz}, \eqref{eq:vanz} and \eqref{eq:zanz} into the r.h.s.\ of \eqref{eq:phi_pre_gali}.

Expanding $-\Delta \varphi + g(|\varphi|^2)\varphi$ and omitting the dependence on $\omega$ in the r.h.s.'s, except for the ground state $\varphi _{\omega}$, we have,
\begin{align}
-\Delta \varphi[\omega,\mathbf{z}] &= \sum_{\mathbf{m}\in \mathbf{NR}}\mathbf{z}^{\mathbf{m}}(-\Delta \varphi_{\mathbf{m}}),\label{exp:delta}\\
g(|\varphi[\omega,\mathbf{z}]|^2)\varphi[\omega,\mathbf{z}]&=\sum_{\mathbf{m} \in \mathbf{NR}} \mathbf{z}^{\mathbf{m}}g(\varphi_{ \omega}^2)\varphi_{\mathbf{m}}+\sum_{\mathbf{m} \in \mathbf{NR}\setminus\{0\}} \mathbf{z}^{\mathbf{m}}\(g'(\varphi_{ \omega}^2)\varphi_{ \omega}^2(\varphi_{\mathbf{m}}+\varphi_{\overline{\mathbf{m}}}) + g_{\mathbf{m}}\)+I,\label{exp:g}
\end{align}
where
\begin{align}
&g_{\mathbf{m}}=g'(\varphi_{ \omega}^2)
\sum_{\substack{\mathbf{m}^1+\mathbf{m}^2=\mathbf{m}\\ \mathbf{m}^1,\mathbf{m}^2\neq 0}}
\(\varphi_{\mathbf{m}^1}\varphi_{\overline{\mathbf{m}^2}}\varphi_{ \omega}
+\varphi_{\mathbf{m}^1}\varphi_{\mathbf{m}^2}\varphi_{ \omega}
+\varphi_{\overline{\mathbf{m}^1}}\varphi_{\mathbf{m}^2}\varphi_{ \omega}+\sum_{\substack{\mathbf{m}^{11}+\mathbf{m}^{12}=\mathbf{m}^1\\ \mathbf{m}^{11},\mathbf{m}^{12}\neq 0}}\varphi_{\mathbf{m}^{11}}\varphi_{\overline{\mathbf{m}^{12}}}\varphi_{\mathbf{m}^2}\)\nonumber\\
&+\sum_{n=2}^{\infty}\frac{1}{n!}g^{(n)}(\varphi_{ \omega}) \sum_{\substack{\mathbf{m}^1+\cdots+\mathbf{m}^{n+1}=\mathbf{m}\\ \mathbf{m}^1,\cdots,\mathbf{m}^n\neq 0 }} \prod_{j=1}^n\(\varphi_{\omega}\varphi_{\mathbf{m}^j}+\varphi_{\omega}\varphi_{\overline{\mathbf{m}^j}}+\sum_{\substack{\mathbf{m}^{j1}+\mathbf{m}^{j2}=\mathbf{m}^j\\ \mathbf{m}^{j1},\mathbf{m}^{j2}\neq 0}}\varphi_{\mathbf{m}^{j1}}\varphi_{\mathbf{m}^{j2}}\)\varphi_{\mathbf{m}^{n+1}}.\label{expand:g}
\end{align}
and $I=I[\omega,\mathbf{m}]$ is a collection of remainders.
\begin{remark}
In \eqref{expand_theta_phi}, \eqref{expand_v_phi} and \eqref{exapnd_z_phi} below, we use $I$ with the same meaning.
\end{remark}
Notice that in the 1st sum in the second line of \eqref{expand:g}, for $n>\|\mathbf{m}\|$, the set of $\mathbf{m}^1,\cdots,\mathbf{m}^{n+1}$ satisfying the condition of the 2nd sum is empty.
Thus, the 1st sum is just a finite sum and not an infinite series.
Furthermore, notice that all terms appearing in \eqref{expand:g} consist  of $\varphi_{\mathbf{n}}$ with $\|\mathbf{n}\|< \|\mathbf{m}\|$.
In this sense, $g_{\mathbf{m}}$ is a "known" term.

We next expand the terms $\widetilde{\theta}\varphi$ and $\frac{1}{2}\widetilde{v_l}x_l\varphi$ for $l=1,2,3$.
\begin{align}
\widetilde{\theta}(\omega,\mathbf{z})\varphi[\omega,\mathbf{z}]&=
\sum_{\mathbf{m}\in \Lambda_0}\mathbf{z}^{\mathbf{m}}\widetilde{\theta}_{\mathbf{m}}\varphi_{ \omega}
+\sum_{\mathbf{m} \in \mathbf{NR}}\mathbf{z}^{\mathbf{m}}\omega \varphi_{\mathbf{m}}
+\sum_{\mathbf{m}\in \mathbf{NR}}\sum_{\substack{\mathbf{m}^1+\mathbf{m}^2=\mathbf{m}\label{expand_theta_phi}\\ \mathbf{m}^1,\mathbf{m}^2\neq 0}}
\mathbf{z}^{\mathbf{m}}\widetilde{\theta}_{\mathbf{m}^1}\varphi_{\mathbf{m}^2}+I,\\
\frac{1}{2}\widetilde{v_l}(\omega,\mathbf{z})x_l\varphi[\omega,\mathbf{z}]&=\sum_{\mathbf{m}\in \Lambda_0}\mathbf{z}^{\mathbf{m}}\frac{1}{2}\widetilde{v}_{l\mathbf{m}}x_l\varphi_\omega
+\sum_{\mathbf{m}\in \mathbf{NR}}\sum_{\substack{\mathbf{m}^1+\mathbf{m}^2=\mathbf{m}\\ \mathbf{m}^1,\mathbf{m}^2\neq 0}}
\mathbf{z}^{\mathbf{m}}\frac{1}{2}\widetilde{v}_{l,\mathbf{m}^1}x_l\varphi_{\mathbf{m}^2}+I.\label{expand_v_phi}
\end{align}
The 3rd term in r.h.s.\ of  \eqref{expand_theta_phi} and the 2nd term in r.h.s.\ of \eqref{expand_v_phi} are known terms.

Expanding $-\im  D_{\mathbf{z}}\varphi\widetilde{\mathbf{z}}$, we have
\begin{align}
\nonumber &-\im  D_{\mathbf{z}}\varphi\widetilde{\mathbf{z}}= -\sum_{\mathbf{m}\in \mathbf{NR}}\mathbf{z}^{\mathbf{m}}\sum_{k=1}^{N}\(\sum_{\mathbf{m}^1+\mathbf{m}^2=\mathbf{m}+\mathbf{e}^{k+}}m^1_{k+}\lambda_{k,\mathbf{m}^2}\varphi_{\mathbf{m}^1}-\sum_{\mathbf{m}^1+\overline{\mathbf{m}^2}=\mathbf{m}+\mathbf{e}^{k-}}m^1_{k-}\lambda_{k,\mathbf{m}^2}\varphi_{\mathbf{m}^1}\)+I\\&
=-\sum_{\mathbf{m}\in \mathbf{NR}}\mathbf{z}^{\mathbf{m}}\lambda(\omega,\mathbf{m})\varphi_{\mathbf{m}}-\sum_{k=1}^N\(\sum_{\mathbf{m}\in \mathbf{NR}\setminus\{\mathbf{e}^{k+}\}} \mathbf{z}^{\mathbf{m}}\lambda_{k,\mathbf{m}}\xi_{j+}-\sum_{\mathbf{m}\in \mathbf{NR}\setminus\{\mathbf{e}^{k-}\}}\mathbf{z}^{\mathbf{m}}\lambda_{k,\overline{\mathbf{m}}}\xi_{j-}\label{exapnd_z_phi}\)\\&\quad
 -\sum_{\mathbf{m}\in \mathbf{NR}}\mathbf{z}^{\mathbf{m}}\sum_{k=1}^{N}\(\sum_{\substack{\mathbf{m}^1+\mathbf{m}^2=\mathbf{m}+\mathbf{e}^{k+}\\ 2\leq \|\mathbf{m}^1\|<\|\mathbf{m}\|}}m^1_{k+}\lambda_{k,\mathbf{m}^2}\varphi_{\mathbf{m}^1}-\sum_{\substack{\mathbf{m}^1+\overline{\mathbf{m}^2}=\mathbf{m}+\mathbf{e}^{k-}\\ 2\leq \|\mathbf{m}^1\|<\|\mathbf{m}\|}}m^1_{k-}\lambda_{k,\mathbf{m}^2}\varphi_{\mathbf{m}^1}\)+I.\nonumber
\end{align}
The last line (except $I$) are known terms.

We collect all known terms in what we denote $K_{\mathbf{m}}$.
That is,
\begin{align}
K_{\mathbf{m}}[\omega]=&g_{\mathbf{m}}
+\sum_{\substack{\mathbf{m}^1+\mathbf{m}^2=\mathbf{m}\\ \mathbf{m}^1,\mathbf{m}^2\neq 0}}
\(\widetilde{\theta}_{\mathbf{m}^1}\varphi_{\mathbf{m}^2}+\frac{1}{2}\widetilde{v}_{l,\mathbf{m}^1}x_l\varphi_{\mathbf{m}^2}\)\label{eq:Km}\\
&-\sum_{k=1}^{N}\(\sum_{\substack{\mathbf{m}^1+\mathbf{m}^2=\mathbf{m}+\mathbf{e}^{k+}\\ 2\leq \|\mathbf{m}^1\|<\|\mathbf{m}\|}}m^1_{k+}\lambda_{k,\mathbf{m}^2}\varphi_{\mathbf{m}^1}-\sum_{\substack{\mathbf{m}^1+\overline{\mathbf{m}^2}=\mathbf{m}+\mathbf{e}^{k-}\\ 2\leq \|\mathbf{m}^1\|<\|\mathbf{m}\|}}m^1_{k-}\lambda_{k,\mathbf{m}^2}\varphi_{\mathbf{m}^1}\).\nonumber
\end{align}
Collecting the coefficients of $\mathbf{z}^{\mathbf{m}}$, for all $\mathbf{m}\in \mathbf{NR}$ we impose
\begin{align}
0=&\(-\Delta +\omega + g(\varphi_{ \omega}^2)+g'(\varphi_{ \omega}^2)\varphi_{ \omega}^2\) \varphi_{\mathbf{m}} + g'(\varphi_{ \omega}^2)\varphi_{ \omega}^2 \varphi_{\overline{\mathbf{m}}}-\lambda(\omega,\mathbf{m})\varphi_{\mathbf{m}}  \nonumber\\&+\widetilde{\theta}_{\mathbf{m}}\varphi_{ \omega} +\frac{1}{2}\widetilde{\mathbf{v}}_{\mathbf{m}}\cdot x \varphi_{ \omega}-\sum_{k=1}^N \(\lambda_{k,\mathbf{m}}\xi_{k+}-\lambda_{k,\overline{\mathbf{m}}}\xi_{k-}\)+K_{\mathbf{m}},\label{eq:funrp}
\end{align}
where $\widetilde{\theta}_{\mathbf{m}}=0$ and $\widetilde{\mathbf{v}}_{\mathbf{m}}=0$ for $\mathbf{m}\not\in \Lambda_0$, $\lambda_{j,\mathbf{m}}=0$ for $\mathbf{m}\not\in \Lambda_j\cap\{\|\mathbf{m}\|\geq 2\}$, the terms with $g'$ are absent when $\mathbf{m}=0$ and  the 1st (resp.\ 2nd) term in $\sum_{k=1}^{N}$ is absent when $\mathbf{m}=\mathbf{e}^{j+}$ ($\mathbf{e}^{j-}$).

We first check that the root cases $\mathbf{m}=0$ and $\mathbf{e}^{j\pm}$ are satisfied for our initial choice given in \eqref{eq:anzphi}.

\begin{claim}
$\varphi_\mathbf{m}=\varphi_{ \omega}\in C^\infty (\mathcal{O}, \Sigma ) $ solves \eqref{eq:funrp} for the case $\mathbf{m}=0$.
\end{claim}

\begin{proof}
In this case \eqref{eq:funrp} is reduced to \eqref{eq:sp}.
\end{proof}

\begin{claim}
Setting $\lambda_{j\mathbf{e}^j}(\omega)=\lambda_j(\omega)$, $\varphi_{\mathbf{e}_{j\pm}}=\xi_{j\pm}[\omega]\in C^\infty (\mathcal{O}, \Sigma ) $ solves \eqref{eq:funrp} for $\mathbf{m}=\mathbf{e}_{j\pm}$.
\end{claim}

\begin{proof}
In this case, the 2nd line of \eqref{eq:funrp} vanishes and combining \eqref{eq:funrp} for $\mathbf{m}=\mathbf{e}^{j+}$ and $\mathbf{m}=\mathbf{e}^{j-}$, we obtain
$
\(\mathcal{H}_{\omega}-\lambda_j(\omega)\)
\xi_j[\omega]
=0,
$
which is the definition of $\xi_j[\omega]$.
\end{proof}

In the following, we assume that we have determined $K_{\mathbf{m}}$, that is,  we have determined $\varphi_{\mathbf{n}}$, $\widetilde{\theta}_{\mathbf{n}}$, $\widetilde{\mathbf{v}}_{\mathbf{n}}$ and $\lambda_{k,\mathbf{n}}$ for all $\mathbf{n}$ with $\|\mathbf{n}\|< \|\mathbf{m}\|$.
We start from the case $\mathbf{m}\in \Lambda_{0}$.

\begin{claim}\label{claim:Lambda0}
Let $\mathbf{m}\in \Lambda_0$.
Then, we can choose $\widetilde{\theta}_{\mathbf{m}}, \widetilde{\theta}_{\overline{\mathbf{m}}}$, $\widetilde{\mathbf{v}}_{\mathbf{m}}$ and $\widetilde{\mathbf{v}}_{\overline{\mathbf{m}}}$ so that we can solve \eqref{eq:funrp} with $\mathbf{m}$ and $\overline{\mathbf{m}}$ for $\varphi_{\mathbf{m}}[\omega]$ and $\varphi_{\overline{\mathbf{m}}}[\omega]$.
Furthermore, $\varphi_{\mathbf{m}}[\omega]$ and $\varphi_{\overline{\mathbf{m}}}[\omega]$ are   in $C ^\infty (\mathcal{O},\Sigma )$, restricting $\mathcal{O}$ if necessary.
\end{claim}

\begin{proof}
In this case, we can rewrite \eqref{eq:funrp} as
\begin{align*}
0=\(-\Delta +\omega + g(\varphi_{ \omega}^2)+g'(\varphi_{ \omega}^2)\varphi_{ \omega}^2\) \varphi_{\mathbf{m}} + g'(\varphi_{ \omega}^2)\varphi_{ \omega}^2 \varphi_{\overline{\mathbf{m}}}+\widetilde{\theta}_{\mathbf{m}}\varphi_{ \omega} +\frac{1}{2}\widetilde{\mathbf{v}}_{\mathbf{m}}\cdot x \varphi_{ \omega}-\lambda(\omega,\mathbf{m})\varphi_{\mathbf{m}} +K_{\mathbf{m}}.
\end{align*}
Since $\overline{\mathbf{m}}\in \Lambda_0$, adding and subtracting the above equations for $\mathbf{m}$ and $\overline{\mathbf{m}}$, we have
\begin{align}
0&=L_{\omega,+}\varphi_{\mathbf{m}+}-\lambda(\omega,\mathbf{m})\varphi_{\mathbf{m}-}+\widetilde{\theta}_{\mathbf{m}+}\varphi_{\omega}+\frac{1}{2}\widetilde{\mathbf{v}}_{\mathbf{m}+}\cdot  x \varphi_{ \omega}+K_{\mathbf{m}+},\label{eq:Lambda01}\\
0&=L_{\omega,-}\varphi_{\mathbf{m}-}-\lambda(\omega,\mathbf{m})\varphi_{\mathbf{m}+}+\widetilde{\theta}_{\mathbf{m}-}\varphi_{ \omega}+\frac{1}{2}\widetilde{\mathbf{v}}_{\mathbf{m}-}\cdot x\varphi_{ \omega} +K_{\mathbf{m}-},\label{eq:Lambda02}
\end{align}
where $L_{\omega,-}=-\Delta + \omega +g(\varphi_\omega^2)$ and $x_{\mathbf{m}\pm}=x_{\mathbf{m}}\pm x_{\overline{\mathbf{m}}}$ for $x=\varphi,\widetilde{\theta}, \widetilde{\mathbf{v}}$ and $K$.
To solve \eqref{eq:Lambda01} and \eqref{eq:Lambda02}, since $\ker L_{\omega,+} = \mathrm{span}\{\partial_{x_l}\varphi_\omega\ |\ l=1,2,3\}$ from (H1) and since, from the fact that $\varphi_\omega$ is positive,  $L_{\omega,-}=\mathrm{span}\{\varphi _{\omega}\}$,  we set $\widetilde{\theta}_{\mathbf{m}+}=0$, $\widetilde{\mathbf{v}}_{\mathbf{m}-}=0$ and choose $\widetilde{\theta}_{\mathbf{m}-}$ and $\widetilde{v}_{l\mathbf{m}+}$ to satisfy
\begin{align}
\frac{1}{2}\widetilde{v}_{l\mathbf{m}+} \<x_l \varphi_{ \omega},\partial_{x_l}\varphi_{ \omega}\>+\<K_{\mathbf{m}+},\partial_{x_l}\varphi_{ \omega}\>-\lambda(\omega,\mathbf{m})\<\varphi_{\mathbf{m}-},\partial_{x_l}\varphi_{ \omega}\>&=0,\label{eq:Lambda03}\\
\widetilde{\theta}_{\mathbf{m}-}\<\varphi_{ \omega},\varphi_\omega\> +\<K_{\mathbf{m}-},\varphi_\omega\>-\lambda(\omega,\mathbf{m})\<\varphi_{\mathbf{m}+},\varphi_\omega\>&=0.\label{eq:Lambda04}
\end{align}
From $\<x_l\varphi_{\omega},\partial_{x_l}\varphi_{ \omega}\>=-\frac{1}{2}\|\varphi_{ \omega}\|_{L^2}^2$, we can always solve \eqref{eq:Lambda03} and \eqref{eq:Lambda04} w.r.t.\ $\widetilde{v}_{l\mathbf{m}+}$ and $\widetilde{\theta}_{\mathbf{m}-}$ for given $\varphi_{\mathbf{m}\pm}$.
\begin{remark}
When $\mathbf{m}=\overline{\mathbf{m}}$, \eqref{eq:Lambda02} is trivial.
Notice that in this case we have $\lambda(\omega,\mathbf{m})=0$.
\end{remark}
Substituting $\widetilde{v}_{l\mathbf{m}+}$ and $\widetilde{\theta}_{\mathbf{m}-}$ given in \eqref{eq:Lambda03} and \eqref{eq:Lambda04},  into \eqref{eq:Lambda01} and \eqref{eq:Lambda02}, we obtain
\begin{align}\label{eq:Lambda05}
\(\begin{pmatrix}
L_{\omega+}& 0 \\ 0 & L_{\omega-}
\end{pmatrix} -\lambda(\omega,\mathbf{m})P_0[\omega]\sigma_1\)
\begin{pmatrix}
\varphi_{\mathbf{m}+}\\\varphi_{\mathbf{m}-}
\end{pmatrix}
=-P_0[\omega]\begin{pmatrix}
K_{\mathbf{m}+}\\ K_{\mathbf{m}-}
\end{pmatrix},
\end{align}
where $P_0[\omega]$ is the orthogonal projection $L^2(\R^3)\to \mathrm{ker}(L_{\omega+})^\perp \oplus \mathrm{ker}(L_{\omega-})^\perp$.
Since $\mathrm{diag}(L_{\omega+}\ L_{\omega-})$ is invertible on $\mathrm{Ran}P_0[\omega]$ and $\lambda(\omega,\mathbf{m})$ is small if $\omega$ is near $\omega_*$, we can take the inverse of the operator in the l.h.s.\ of \eqref{eq:Lambda05}.
Thus, we have solved \eqref{eq:funrp} for   $\mathbf{m}\in \Lambda_0$.
\end{proof}

We next consider the case $\mathbf{m}\in \Lambda_j$.   For $k$ s.t.\ $\mathbf{m}\in \Lambda_k$ set
\begin{align}\label{eq:lambdakm}
\lambda_{k\mathbf{m}}=\<\widetilde{K}_{\mathbf{m}},\xi_k\>, \text{where } \widetilde{K}_{\mathbf{m}}[\omega]  :=
 \begin{pmatrix}
K_{\mathbf{m}} [\omega]\\ K_{\overline{\mathbf{m}}}[\omega]
\end{pmatrix}.
\end{align}
\begin{claim}\label{claim:Lambdaj}
For $\mathbf{m}\in \Lambda_j$.
  we can solve \eqref{eq:funrp}   for $\varphi_{\mathbf{m}}[\omega]$ and $\varphi_{\overline{\mathbf{m}}}[\omega]$.
Furthermore, $\varphi_{\mathbf{m}}[\omega]$ and $\varphi_{\overline{\mathbf{m}}}[\omega]$ are in  $C ^\infty (\mathcal{O},\Sigma )$, restricting $\mathcal{O}$ if necessary.
\end{claim}


\begin{proof}
Combining \eqref{eq:funrp} for $\mathbf{m}$ and $\overline{\mathbf{m}}$, we have
\begin{align}\label{eq:case_m_in_Lamdaj}
0=\(\mathcal{H}_\omega -\lambda(\omega,\mathbf{m})\) \phi_{\mathbf{m}} - \sum_{k:\mathbf{m}\in \Lambda_k}\lambda_{k,\mathbf{m}}\xi_{k}+\sigma_3\widetilde{K}_{\mathbf{m}},\ \text{where}\ \phi_{\mathbf{m}}:=
\begin{pmatrix}
\varphi_{\mathbf{m}}\\ \varphi_{\overline{\mathbf{m}}}
\end{pmatrix}.
\end{align}
Substituting \eqref{eq:lambdakm}, we have
\begin{align}\label{eq:case_m_in_Lamdaj_proj}
0=\(\mathcal{H}_\omega -\lambda(\omega,\mathbf{m})\) \phi_{\mathbf{m}}  +P_j[\omega]\sigma_3\widetilde{K}_{\mathbf{m}},\ \text{where}\ P_{j+}[\omega]:=1-\sum_{k:\mathbf{m}\in \Lambda_k}(\cdot,\sigma_3 \xi_k[\omega])\xi_k[\omega].
\end{align}
Notice that by \eqref{eq:krein}, $P_{j+}[\omega]$ is a projection and one can check $[\mathcal{H}_{\omega},P_{j+}[\omega]]=0$ and $\mathrm{Ran}P_{j+}=\{\sigma_3\xi_k[\omega]\ | k\  \mathrm{s.t.}\ \mathbf{m}\in \Lambda_k\}^{\perp}$.
Therefore, we have
\begin{align*}
\mathrm{Ran}(\mathcal{H}_{\omega_*}-\lambda_j(\omega_*))=\mathrm{Ker}\(\mathcal{H}_{\omega_*}^*-\lambda_j(\omega_*)\)^{\perp}=\mathrm{Ran}P_{j+}[\omega_*],
\end{align*}
where we have used the fact that $\mathrm{Ran}\(\mathcal{H}_{\omega_*}-\lambda_j(\omega_*)\)$ is closed, $\mathcal{H}_{\omega}^*=\sigma_3\mathcal{H}_\omega \sigma_3$ and $\mathrm{ker}\mathcal{H}_{\omega_*}=\{\xi_k[\omega]\ | k\  \mathrm{s.t.}\ \mathbf{m}\in \Lambda_k\}$.
Therefore, the inverse of  $\left.\(\mathcal{H}_{\omega_*}-\lambda_j(\omega_*)\)\right|_{P_{j+}[\omega_*]}$ exists. 
Now, set $U_j[\omega]$ by
\begin{align}
U_j[\omega]:=\(P_{j+}[\omega]P_{j+}[\omega_*]+\(1-P_{j+}[\omega]\)\(1-P_{j+}[\omega_*]\)\)\(1-\(P_{j+}[\omega]-P_{j+}[\omega_*]\)^2\)^{-1/2},
\end{align}
then, we have $U_j[\omega_*]=1$ and $P_{j+}[\omega]=U_j[\omega]P_{j+}[\omega_*]U_j[\omega]^{-1}$ (see 1.6.7 of \cite{KatoPertBook}).
Applying $U[\omega]^{-1}$ to \eqref{eq:case_m_in_Lamdaj_proj}, we have
\begin{align}
\(\mathcal{H}_{\omega_*}-\lambda_j(\omega_*)+V[\omega]\)\(U[\omega]^{-1}\phi_{\mathbf{m}}\)+P_{j+}[\omega_*]U[\omega]^{-1}\sigma_3\widetilde{K}_{\mathbf{m}}=0,
\end{align}
where
\begin{align*}
V[\omega]=P_{j+}[\omega_*] \(\(U[\omega]^{-1}-1\)\(\mathcal{H}_{\omega}-\lambda_j(\omega,\mathbf{m})\)+\(\mathcal{H}_\omega-\mathcal{H}_{\omega_*}-\lambda_j(\omega,\mathbf{m})+\lambda_j(\omega_+)\)\)P_{j+}[\omega_*].
\end{align*}
Thus, we have
\begin{align}
\phi_{\mathbf{m}}=-U[\omega]\sum_{n=0}^{\infty}\(-(\mathcal{H}_{\omega_*}-\lambda_j(\omega_*))^{-1}V[\omega]\)^nP_{j+}[\omega_*]U[\omega]^{-1}\sigma_3\widetilde{K}_{\mathbf{m}}.
\end{align}
Notice that since $V[\omega_*]=0$, the series converges near $\omega=\omega_*$.
The smoothness of $\phi_{\mathbf{m}}$ w.r.t.\ $\omega$ follows from the above expression.
\end{proof}

\begin{claim}\label{claim:other}
Let $\mathbf{m}\in \mathbf{NR}\setminus\{0\}$ with $\mathbf{m}\not\in\Lambda_0$, $\mathbf{m}\not\in \Lambda_j$ and $\overline{\mathbf{m}}\not\in \Lambda_j$.
Then, there exist  $\varphi_{\mathbf{m}}[\omega]$ and $\varphi_{\overline{\mathbf{m}}}[\omega]$ satisfying \eqref{eq:funrp} and $C ^\infty (\mathcal{O},\Sigma )$, restricting $\mathcal{O}$ if necessary.
\end{claim}

\begin{proof}
In this case, \eqref{eq:funrp} can be rewritten as
\begin{align}
\(\mathcal{H}_\omega-\lambda(\omega,\mathbf{m})\)\phi_{\mathbf{m}}+\sigma_3\widetilde{K}_{\mathbf{m}}=0.
\end{align}
Since $\mathcal{H}_\omega-\lambda(\omega,\mathbf{m})$ has a bounded inverse for $\omega$ sufficiently near $\omega_*$, we can solve the above w.r.t.\  $\phi_{\mathbf{m}}={}^t(\varphi_{\mathbf{m}}\ \varphi_{\overline{\mathbf{m}}})$.
\end{proof}

\subsubsection*{2nd step}
For the last step,  for $\varphi = \varphi [\omega ,\mathbf{z}] $  we consider
\begin{align}\label{def:tildeR}
\widetilde{\mathcal{R}}:=&-\Delta \varphi + g(|\varphi|^2)\varphi \nonumber\\&+ \(\omega+\sum_{\mathbf{m}\in \Lambda_0}\mathbf{z}^{\mathbf{m}}\widetilde{\theta}_{\mathbf{m}}\)\varphi +\frac{1}{2}\sum_{\mathbf{m}\in \Lambda_0}\mathbf{z}^{\mathbf{m}}\widetilde{\mathbf{v}}_{\mathbf{m}}\cdot x \varphi -\im \sum_{j=1}^ND_{z_j}\varphi \(-\im \sum_{\mathbf{m}\in \Lambda_j}\mathbf{z}^{\mathbf{m}}\widetilde{\lambda}_{j,\mathbf{m}}\) .
\end{align}
Since all the coefficients of $\mathbf{z}^{\mathbf{m}}$ with $\mathbf{m}\in \mathbf{NR}$ in the r.h.s.\ of \eqref{def:tildeR} are $0$, we have
$\|\widetilde{R}\|_{\Sigma^s}\lesssim_s \sum_{\mathbf{m} \in \mathbf{R}_{\mathrm{min}}}|\mathbf{z}^{\mathbf{m}}|$.
From \eqref{eq:phi_pre_gali} and \eqref{eq:thetaanz}--\eqref{eq:zanz} we have
\begin{align*}
\mathcal{R}=\widetilde{\mathcal{R}}+\widetilde{\theta}_{\mathcal{R}}\varphi + \widetilde{\omega}_{\mathcal{R}}\im \partial_{\omega}\varphi - \im \sum_{l=1}^3 \widetilde{y}_{l\mathcal{R}}\partial_{x_l}\varphi +\frac{1}{2}\widetilde{\mathbf{v}}_{\mathcal{R}}\cdot x\varphi-\im D_{\mathbf{z}}\varphi \widetilde{\mathbf{z}}_{\mathcal{R}}.
\end{align*}
To make $\mathcal{R}$ satisfy \eqref{R:orth}, $\widetilde{\theta}_{\mathcal{R}}$, $\widetilde{\omega}_{\mathcal{R}}$, $\widetilde{\mathbf{y}}_{\mathcal{R}}$, $\widetilde{\mathbf{v}}_{\mathcal{R}}$ and $\widetilde{\mathbf{z}}_{\mathcal{R}}$ need to satisfy the following equation:
\begin{align}
\begin{pmatrix}
\<\widetilde{\mathcal{R}},\im \varphi\>\\
\<\widetilde{\mathcal{R}},\partial_{\omega}\varphi\>\\
\<\widetilde{\mathcal{R}},\partial_{x_1}\varphi\>\\
\vdots\\
\<\widetilde{\mathcal{R}},\im x_l \varphi\>\\
\<\widetilde{\mathcal{R}},\partial_{z_{1R}}\varphi\>\\
\vdots\\
\<\widetilde{\mathcal{R}},\partial_{z_{NI}}\varphi\>
\end{pmatrix}
+
\mathcal{A}[\omega,\mathbf{z}]
\begin{pmatrix}
\widetilde{\theta}_{\mathcal{R}}\\
\widetilde{\omega}_{\mathcal{R}}\\
\widetilde{y}_{1\mathcal{R}}\\
\vdots\\
\widetilde{v}_{3\mathcal{R}}\\
\widetilde{z}_{1R\mathcal{R}}\\
\vdots\\
\widetilde{z}_{NI\mathcal{R}}
\end{pmatrix}
=0,\label{solvetilde}
\end{align}
for an appropriate  matrix $\mathcal{A}[\omega,\mathbf{z}]$   obtained  substituting the orthogonality condition.
From (H2) it is well known and elementary to see  that $\mathcal{A}[\omega ,\mathbf{z}]$ is invertible for $\omega=\omega_*$ and $\mathbf{z}=0$.
Thus, if $|\omega-\omega_*|+\|\mathbf{z}\|$ is sufficiently small, we can solve the above equation and $\mathcal{R}$ will satisfy \eqref{R:orth}.
Finally, the estimate \eqref{eq:zanz}  follows  from \eqref{solvetilde} and  $\|\widetilde{R}\|_{\Sigma^s}\lesssim_s \sum_{\mathbf{m} \in \mathbf{R}_{\mathrm{min}}}|\mathbf{z}^{\mathbf{m}}|$.
\end{proof}

\begin{remark}
The proof of Proposition \ref{prop:rp_pre_galilei} is rather involved because we are trying to construct $\varphi [\omega ,\mathbf{z}]$ in a neighborhood of $\omega_*$.
However, for the Fermi Golden Rule assumption (H7), it suffices to know $G_{\mathbf{m}}$, which can be constructed in much simple manner because we only have to consider $\omega=\omega_*$.
Indeed, in step 1, Claim \ref{claim:Lambda0}, we have
\begin{align*}
\widetilde{v}_{l\mathbf{m}+}&=4\|\varphi_{\omega_*}\|_{L^2}^{-2}\<K_{\mathbf{m}+}[\omega_*],\partial_{x_l}\varphi_{\omega_*}\>,\\
\widetilde{\theta}_{\mathbf{m}-}&=-\|\varphi_{ \omega_*}\|_{L^2}^{-2}\<K_{\mathbf{m}-}[\omega_{*}],\varphi_{ \omega_*}\>,
\end{align*}
and
\begin{align*}
\varphi_{\mathbf{m}\pm}=-L_{\omega_* \pm}^{-1}P_{0\pm}[\omega_*]\sigma_3K_{\mathbf{m}\pm}[\omega_*],
\end{align*}
where $P_{0\pm}$ are the orthogonal projection on $\mathrm{Ker}(L_{\omega\pm})$.
Similarly, in Claim \ref{claim:Lambdaj}, we have $\lambda_{k\mathbf{m}}$ given by \eqref{eq:lambdakm} and
\begin{align*}
\phi_{\mathbf{m}}[\omega_*]=(\mathcal{H}_{\omega_*}-\lambda_j(\omega_*))^{-1}P_j[\omega_*]\widetilde{K}_{\mathbf{m}}[\omega_*],
\end{align*}
and in Claim \ref{claim:other}, $\phi_{\mathbf{m}}=-(\mathcal{H}_{\omega_*}-\lambda(\omega_*,\mathbf{m}))^{-1}\sigma_3\widetilde{K}_{\mathbf{m}}[\omega_*]$.
Thus, we can inductively define $\widetilde{G}_{\mathbf{m}}$ in a very explicit manner using the above formulas by the r.h.s.\ of \eqref{eq:Km} with $\omega=\omega_*$ and $\mathbf{m}\in \mathbf{R}_{\mathrm{min}}$.
Finally, since the higher order correction of $\varphi[\omega,\mathbf{z}]$ only affects $\widetilde{R}_1$, we can take a projection of $\widetilde{G}_{\mathbf{m}}$ as step 2, but $\varphi$ replaced by $\varphi_{\omega_*}$ to obtain ${G}_{\mathbf{m}}$.

\end{remark}

\section{Modulation}

For $(\theta,\mathbf{y},\mathbf{v})\in \R^{1+3+3}$, we define  the Galilean transformations (with gauge rotation) by
\begin{align*}
\(G_{\theta,\mathbf{y},\mathbf{v}}u\)(t,x):=e^{\im \theta} e^{\im \frac{1}{2}\mathbf{v}\cdot (x-\mathbf{y})}u(t,x-\mathbf{y}).
\end{align*}

It is well known that the  NLS \eqref{scalarNLS} in invariant under Galilean transformations.
\begin{lemma}\label{lem:Gal}
Suppose $u$ satisfies
\begin{align*}
\im \partial_t u = -\Delta u + g(|u|^2)u -r,
\end{align*}
for some $r=r(t,x)$.
Then, for any $(\theta,\mathbf{y},\mathbf{v})\in \R^{1+3+3}$, $v(t,x):=\(G_{\theta + \frac{1}{4}|\mathbf{v}|^2t,\mathbf{y}+\mathbf{v}t,\mathbf{v}}u\)(t,x)$ solves
\begin{align*}
\im \partial_t v = -\Delta v + g(|v|^2)v - G_{\theta + \frac{1}{4}|\mathbf{v}|^2t,\mathbf{y}+\mathbf{v}t,\mathbf{v}}r.
\end{align*}
\end{lemma}

\begin{proof}
See, e.g.\ Chapter 5 of \cite{LPBook}.
\end{proof}

We extend the refined profile $\varphi[\omega,\mathbf{z}]$ given in Proposition \ref{prop:rp_pre_galilei} by Galilean and gauge symmetry,
\begin{align}\label{def:rp}
\varphi[\theta,\varpi,\mathbf{y},\mathbf{v},\mathbf{z}]:=G_{\theta,\mathbf{y},\mathbf{v}}\varphi[\omega,\mathbf{z}],\ \text{where}\ \varpi=\omega-\omega_*
\end{align}
We also introduce the variable $\Theta:=(\theta,\varpi,\mathbf{y},\mathbf{v},\mathbf{z})$ and write
$\varphi[\Theta]:=\varphi[\theta,\varpi,\mathbf{y},\mathbf{v},\mathbf{z}]$.
Notice that we have $\varphi[0]=\varphi_{\omega_*}$.

In the following, for a smooth function $F$ of $\Theta$ (in particular $\varphi$), we write
\begin{align*}
DF[\Theta]\Xi:=\left.\frac{d}{d\epsilon}\right|_{\epsilon=0}F[\Theta+\epsilon\Xi].
\end{align*}
The 2nd derivative w.r.t.\ $\Theta$ will be  expressed by $D^2 F[\Theta](\Xi_1,\Xi_2)$.
We define $D_{\mathbf{y}}$ and $D_{\mathbf{v}}$ similarly.
Recall that we have already defined $D_{\mathbf{z}}$.
\begin{proposition}\label{prop:full_rp}
For $     \varphi = \varphi [\Theta]$,  we have
\begin{align}\label{eq:full_rp}
\im D\varphi  \widetilde{\Theta}+ \mathcal{R} = \(-\Delta+\omega_{*}\) \varphi + g(|\varphi|^2)\varphi,
\end{align}
where
\begin{align}\label{def:tildeTheta}
\widetilde{\Theta}[\varpi,\mathbf{v},\mathbf{z}]=\(\frac{1}{4}|\mathbf{v}|^2+\widetilde{\theta}[\omega_*+\varpi,\mathbf{z}]-\omega_*,\widetilde{\omega}[\omega_*+\varpi,\mathbf{z}],\mathbf{v}+\widetilde{\mathbf{y}}[\omega_*+\varpi,\mathbf{z}],\widetilde{\mathbf{v}}[\omega_*+\varpi,\mathbf{z}],\widetilde{\mathbf{z}}[\omega_*+\varpi,\mathbf{z}]\)
\end{align} 
and
\begin{align}\label{def:Rgali}
\mathcal{R}=\mathcal{R}[\Theta]=\mathcal{R}[\theta,\omega_*+\varpi,\mathbf{y},\mathbf{v},\mathbf{z}]=G_{\theta,\mathbf{y},\mathbf{v}}\mathcal{R}[\omega_*+\varpi,\mathbf{z}],
\end{align}
for   $x[\omega,\mathbf{z}]$ given in Proposition \ref{prop:rp_pre_galilei} for $x=\widetilde{\theta},  \widetilde{\mathbf{v}}, \widetilde{\mathbf{z}}$  and $\mathcal{R}$.
Furthermore,
\begin{align}\label{eq:iR_in_cont}
\forall \Xi \in \R^{1+1+3+3}\times \C^N \text{    we have } \<\mathcal{R}[\Theta],D\varphi[\Theta]\Xi\>=0.
\end{align}
\end{proposition}

\begin{proof}
From Proposition \ref{prop:rp_pre_galilei}, \eqref{def:rp} and setting $\left.\frac{d}{dt}\right|_{t=0}{\mathrm{x}}=\widetilde{\mathrm{x}}$, for $\mathrm{x}=\theta,\omega,\mathbf{y},\mathbf{v}$ and $\mathbf{z}$, we have
\begin{align*}
\left.\im \partial_t\right|_{t=0} \varphi&=\left.\im \partial_t\right|_{t=\theta=0,\mathbf{y}=\mathbf{v}=0} G_{\theta,\mathbf{y},\mathbf{v}} \varphi[\omega,\mathbf{z}]=-\widetilde{\theta}\varphi+\im \widetilde{\omega}\partial_{\omega}\varphi -\im \widetilde{\mathbf{y}}\cdot\nabla_x  \varphi-\frac{1}{2}x\cdot \widetilde{\mathbf{v}}\varphi+\im D_{\mathbf{z}}\varphi\\&
=-\Delta \varphi + g(|\varphi|^2)\varphi -R[0,\varpi,0,0,\mathbf{z}],
\end{align*}
where $\varphi=\varphi[0,\varpi,0,0,\mathbf{z}]$.
Thus, by Lemma \ref{lem:Gal}, setting $v=\varphi[\theta+\frac{1}{4}|\mathbf{v}|^2t,\varpi,\mathbf{y}+\mathbf{v} t, \mathbf{v},\mathbf{z}]$, we have
\begin{align*}
\left.\im \partial_t\right|_{t=0} v = -\Delta \varphi[\Theta] + g(|\varphi[\Theta]|^2)\varphi[\Theta]-\mathcal{R}[\Theta],
\end{align*}
and
\begin{align*}
\left.\im \partial_t\right|_{t=0} v &= -\(\widetilde{\theta}+\frac{1}{4}|\mathbf{v}|^2\)\varphi + \im \widetilde{\omega}\partial_{\omega}\varphi - \im (\widetilde{\mathbf{y}}+\mathbf{v})\cdot \nabla_{\mathbf{y}}\varphi-\frac{1}{2}x\cdot \widetilde{\mathbf{v}}\varphi + \im D_{\mathbf{z}}\varphi\\&
=\im D\varphi \widetilde{\Theta} - \omega_*\varphi,
\end{align*}
where $\widetilde{\Theta}$ is given in \eqref{def:tildeTheta}.
Therefore we have \eqref{eq:full_rp}.

Finally, \eqref{eq:iR_in_cont} follows from \eqref{R:orth}, \eqref{def:rp}, \eqref{def:Rgali} and
\begin{align*}
\partial_{\theta}G_{\theta,\mathbf{y},\mathbf{v}}f=\im G_{\theta,\mathbf{y},\mathbf{v}}f, \partial_{y_l}G_{\theta,\mathbf{y},\mathbf{v}}f = - \partial_{x_l} G_{\theta,\mathbf{y},\mathbf{v}}f,\ \partial_{v_l}G_{\theta,\mathbf{y},\mathbf{v}}f=\im \frac{1}{2}(x_l-y_l) G_{\theta,\mathbf{y},\mathbf{v}}f,
\end{align*}
for $l=1,2,3$.
\end{proof}

We set
\begin{align}\label{def:H}
H:=H[\Theta]&:=-\Delta +\omega_*+ \left.\frac{d}{d\epsilon}\right|_{\epsilon=0}g(|\varphi[\Theta] + \epsilon \cdot|^2)(\varphi[\Theta] +\epsilon \cdot)\\&=-\Delta + \omega_{*} + g(|\varphi[\Theta]^2|)+g'(|\varphi[\Theta]|^2)|\varphi[\Theta]|^2+g'(|\varphi[\Theta]^2|\varphi[\Theta])^2\mathrm{C},\nonumber
\end{align}
where, $\mathrm{C}u=\overline{u}$.
Notice that due to the complex conjugation $\mathrm{C}$, $H$ is not $\C$-linear but only $\R$-linear.
Furthermore, we can easily check $\<H u,v\>=\<u,Hv\>$.
Differentiating \eqref{eq:full_rp}, we have
\begin{align}\label{eq:HDphi}
HD\varphi \Xi = \im D^2\varphi(\widetilde{\Theta},\Xi) + \im D\varphi(D\widetilde{\Theta}\Xi)+D\mathcal{R}\Xi.
\end{align}

We define the ``continuous space" around $\varphi[\Theta]$ by
\begin{align*}
\mathcal{H}_{\mathrm{c}}[\Theta]:=\{u\in H^1\ |\ \forall \Xi\in \R^{1+1+3+3}\times \C^N,\ \<\im u, D\varphi[\Theta]\Xi\>=0\}.
\end{align*}

\begin{remark}
The condition \eqref{eq:iR_in_cont} can be rephrased as $\im \mathcal{R}[\Theta]\in \mathcal{H}_{\mathrm{c}}[\Theta]$.
\end{remark}

\begin{remark}
Notice that $\mathcal{H}_{\mathrm{c}}$ is an $\R$  vector space and not a $\C$ vector space.
\end{remark}

\begin{lemma}[Modulation lemma]\label{lem:mod}
There exists $\delta>0$ s.t.\ if
\begin{align}\label{eq:asslemod}
\inf_{\theta,\mathbf{y}}\|u-\varphi[\theta,0,\mathbf{y},0,0]\|_{H^1}<\delta,
\end{align}
then there exists $\Theta(u)=\(\theta(u), \omega(u),  \mathbf{y}(u), \mathbf{v}(u), \mathbf{z}(u)\)$ which depends on $u$ smoothly, such that we have
\begin{align}\label{eq:etaorth}
\eta(u):=u-\varphi[\Theta(u)]\in \mathcal{H}_{\mathrm{c}}[\Theta(u)].
%
\end{align}
Furthermore, we have
\begin{align}\label{eq:modest}
|\varpi(u)|+\|\mathbf{v}(u)\|+\|\mathbf{z}(u)\| + \|\eta(u)\|_{H^1}\lesssim \delta.
\end{align}
\end{lemma}

\begin{proof}
Since this is standard, we skip it.
\end{proof}

From the orbital stability Proposition \ref{prop:os}, for any solution $u$ of \eqref{scalarNLS} with $\|u(0)-\varphi_{\omega_*}\|_{H^1}$ sufficiently small, the assumption of Lemma \ref{lem:mod} is satisfied for all $t\geq 0$.
Thus, we apply Lemma \ref{lem:mod} to $e^{-\im \omega_* t}u(t)$ (which also satisfies \eqref{eq:asslemod}) and obtain $x(t):=x(e^{-\im \omega_* t}u(t))$ for all $t\geq 0$, where  $x=\Theta, \theta,\varpi,\mathbf{y},\mathbf{v}, \mathbf{z}$ and $\eta$.
We will also use $\dot{x}:=\frac{d}{dt}x(e^{-\im \omega_* t}u(t))$ for $x=\Theta,\theta,\varpi,\mathbf{y},\mathbf{v}, \mathbf{z}$.

We now substitute
\begin{align}\label{u=phi+eta}
u=e^{\im \omega_* t}\(\varphi[\Theta]+\eta\),
\end{align}
 into \eqref{scalarNLS}.
Then, from \eqref{eq:full_rp}, we have
\begin{align}\label{eq:modnls}
\im \partial_t \eta + \im D\varphi[\Theta] (\dot{\Theta}-\widetilde{\Theta})=H[\Theta] \eta + F + \mathcal{R}[\Theta],
\end{align}
where
\begin{align}\label{def:F}
F=g(|\varphi+\eta|^2)(\varphi+\eta)-g(|\varphi|^2)\varphi - \left.\frac{d}{d\epsilon}\right|_{\epsilon=0}g(|\varphi+\epsilon \eta|^2)(\varphi+\epsilon\eta).
\end{align}

\section{Proof of Theorem \ref{thm:main}}

We consider the Strichartz space
\begin{align*}
\mathrm{Stz}(I):=L^\infty(I, H^1)\cap L^2(I,W^{1,6}).
\end{align*}

The main estimate of this paper is given by the following Proposition.
\begin{proposition}\label{prop:boot}
There exist $\epsilon_0>0$ and $C_0>0$ s.t.\ if $\epsilon:=\|u(0)-\varphi_{\omega _*}\|_{H^1}<\epsilon_0$ and
\begin{align}\label{eq:boot}
\|\eta\|_{\mathrm{Stz}(0,T)}+\sum_{\mathbf{m} \in \mathbf{R}_{\mathrm{min}}}\|\mathbf{z}^{\mathbf{m}}\|_{L^2(0,T)}\leq C_0 \epsilon,
\end{align}
for some $T>0$, then we have \eqref{eq:boot} with $C_0$ replaced by $C_0/2$.
\end{proposition}

\begin{proof}[Proof of Theorem \ref{thm:main}]
By Proposition \ref{prop:boot}, we have \eqref{eq:boot} with $T=\infty$.
Then, by standard argument one can show there exists $\eta_+\in H^1$ s.t.\ $\lim_{t\to\infty}\|\eta(t)-e^{-\im (-\Delta +\omega_*) t}\eta_+\|_{H^1} = 0$.
Thus, we have \eqref{thm:main1}.
\begin{remark}
Recall $u$ is given by \eqref{u=phi+eta}.
\end{remark}

\noindent
If we have \eqref{thm:main2}, we will also have \eqref{thm:main3} by orbital stability.
Thus,  it remains to prove \eqref{thm:main2}.
Let $Q_0(u)= \frac{1}{2}\|u\|_{L^2}^2$ and $Q_l(u)=-\frac{1}{2}\<\im \partial_{x_l}u,u\>$ for $l=1,2,3$.
Then, $Q_l$ is   constant under the flow of \eqref{scalarNLS}.
Since $\|\eta(t)-e^{\im t \Delta}\eta_+\|_{H^1}\to 0$ as $t\to +\infty$, we have
\begin{align*}
Q_l(u(t))-Q_l(\varphi[\Theta(t)])-Q_l(\eta_+)\to 0,
\end{align*}
implying convergence of $Q_l(\varphi[\Theta(t)])$.
Furthermore, by $\mathbf{z}\to 0$, we see that the
$Q_l(\varphi[0,\omega(t),0,\mathbf{v}(t),0])$ converge.
Since $Q_0(\varphi[0,\omega(t),0,\mathbf{v}(t),0])=Q_0(\varphi_{\omega(t)})$, we see that $\omega(t)$ must converge.
Finally, since
$Q_l(\varphi[0,\omega(t),0,\mathbf{v}(t),0])=\frac{1}{2}v_l(t)Q_0(\varphi_{\omega(t)})$, we have the convergence of $v_l(t)$, which gives us the conclusion.
\end{proof}

The remainder of the paper is devoted to the proof of Proposition \ref{prop:boot}.
Before going into the details, we note that from Proposition \ref{prop:os} and \eqref{eq:modest}, we have
\begin{align}\label{Linftybound}
\|\varpi\|_{L^\infty}+\|\mathbf{v}\|_{L^\infty}+\|\mathbf{z}\|_{L^\infty} + \|\eta\|_{L^\infty H^1}\lesssim \epsilon,
\end{align}
and from \eqref{Linftybound} and \eqref{def:tildeTheta}, we have
\begin{align}\label{widetildeLinfty}
\|\widetilde{\Theta}\|_{L^\infty}\lesssim \epsilon.
\end{align}
We give  now  estimates for the term $F$ introduced  in \eqref{def:F}.
\begin{lemma}\label{lem:Fbound}
We have
\begin{align}
\|F\|_{L^2 W^{1,6/5}}\lesssim \epsilon\|\eta\|_{\mathrm{Stz} },\label{F:L2est}\\
\|F\|_{L^1L^{6/5}}\lesssim \|\eta\|_{\mathrm{Stz} }^2.\label{F:L1est}
\end{align}
\end{lemma}

\begin{proof}
We write $g(|u|^2)u=\widetilde{g}(u)$ and we ignore in this proof   complex conjugation, which is irrelevant to the estimates.
Notice that from $g(0)=0$, we have $\widetilde{g}(0)=\widetilde{g}'(0)=\widetilde{g}''(0)=0$.
Then, we have
\begin{align*}
F=\widetilde{g}(\eta)+\int_0^1\int_0^1(1-t)\widetilde{g}^{'''}(s\varphi+t\eta)\varphi \eta^2\,dsdt.
\end{align*}
Thus, from \eqref{eq:ggrowth}, we have
\begin{align}\label{Fpoint1}
|F|&\lesssim \<\eta\>^{2}|\eta|^3+\<\eta\>^{2}|\varphi||\eta|^2.\\
|\nabla_x F|&\lesssim \<\eta\>^{2}|\eta|^2|\nabla_x \eta| +\<\eta\>\<\nabla_x \eta\> |\varphi| |\eta|^2+\<\eta\>^{2}|\nabla_x\varphi||\eta|^2+\<\eta\>^{2}|\varphi| |\eta| |\nabla_x \eta| \nonumber
\end{align}
From H\"older and Sobolev inequalities, we have
\begin{align}\label{Fpoint11}
\| \<\eta\>^{2}|\eta|^3+\<\eta\>^{2}|\eta|^2|\nabla_x \eta|\|_{L^2L^{6/5}} \lesssim \<\|\eta\|_{L^\infty H^1}\>^2\|\eta\|_{L^\infty H^1}^2\|\eta\|_{L^2 W^{1,6}}
\end{align}
Similarly, just replacing one $\eta$ in the above inequality with $\varphi$ or $\nabla_x\varphi$, we have
\begin{align}\label{Fpoint12}
\|\<\eta\>^{2}|\varphi||\eta|^2+\<\eta\>^{2}|\nabla_x\varphi||\eta|^2+\<\eta\>^{2}|\varphi| |\eta| |\nabla_x \eta| \|_{L^2L^{6/5}}
\lesssim \<\|\eta\|_{L^\infty H^1}\>^2\|\eta\|_{L^\infty H^1}\|\eta\|_{L^2 W^{1,6}}
\end{align}
For the remaining term,
\begin{align}\label{Fpoint13}
\|\<\eta\>\<\nabla_x \eta\> |\varphi| |\eta|^2\|_{L^2L^{6/5}}\lesssim \<\|\eta\|_{L^\infty H^1}\>\|\eta\|_{L^\infty H^1}\|\eta\|_{L^\infty W^{1,6}}
\end{align}
Therefore, from \eqref{Linftybound}, \eqref{Fpoint11}, \eqref{Fpoint12} and \eqref{Fpoint13}, we have \eqref{F:L2est}.

For \eqref{F:L2est}, use \eqref{Fpoint1} and we will have the conclusion.
\end{proof}

Taking the inner product $\<\eqref{eq:modnls},D\varphi\Xi\>$ and using \eqref{eq:HDphi} and \eqref{eq:etaorth}, we have
\begin{align*}
\<\im D\varphi \(\dot{\Theta}-\widetilde{\Theta}\),D\varphi\Xi\>=-\<\eta,\im D^2\varphi(\dot{\Theta}-\widetilde{\Theta},\Xi)\>+\<\eta,D\mathcal{R}\Xi\>+\<F,D\varphi\Xi\>.
\end{align*}
Thus, substituting $\Xi=(1,0,0,0,0),\cdots, (0,0,0,0,\mathbf{e}^{N-})$ and using \eqref{Linftybound},  $\|D\mathcal{R}\Xi\|_{L^\infty\Sigma^1}\lesssim \epsilon$ and Lemma \ref{lem:Fbound} , we obtain
\begin{align}\label{Theta:L2}
\|\dot{\Theta}-\widetilde{\Theta}\|_{L^2}\lesssim \epsilon \|\eta\|_{\mathrm{Stz} }+\|\eta\|_{\mathrm{Stz} }^2.
\end{align}
Also, by \eqref{Fpoint1}, taking the $L^\infty$ norm, we have
\begin{align}\label{eq:ThetaLinfty}
\|\dot{\Theta}-\widetilde{\Theta}\|_{L^\infty}\lesssim \epsilon^2.
\end{align}
The estimate \eqref{eq:ThetaLinfty} combined with \eqref{widetildeLinfty} implies
\begin{align}\label{eq:Thetadot}
\|\dot{\Theta}\|_{L^\infty}\lesssim \epsilon.
\end{align}

\subsection{Estimate of continuous variables}

To use the properties and estimates of $\mathcal{H}_{\omega_*}$, we rewrite \eqref{eq:modnls} as
\begin{align}\label{eq:modnlsvec}
\im \partial_t \zeta + \im D\phi (\dot{\Theta}-\widetilde{\Theta}) = \mathcal{H}[\Theta]\zeta + \sigma_3(\mathfrak{F}+\mathfrak{R}),
\end{align}
where
\begin{align*}
\zeta=\begin{pmatrix}
\eta\\ \overline{\eta}
\end{pmatrix},\
\phi[\Theta]=
\begin{pmatrix}
\varphi[\Theta]\\
\overline{\varphi[\Theta]}
\end{pmatrix},\
\mathfrak{F}=
\begin{pmatrix}
F\\ \overline{F}
\end{pmatrix},\
\mathfrak{R}=
\begin{pmatrix}
\mathcal{R}\\ \overline{\mathcal{R}}
\end{pmatrix},
\end{align*}
and $\mathcal{H}[\Theta]:=\mathcal{H}[\omega_*,\varphi[\Theta]]$.
We set
$
\widetilde{\mathcal{H}}_{\mathrm{c}}[\Theta]:=\{ w={}^t(u\ \overline{u})\in H^1\ |\ u\in \mathcal{H}_{\mathrm{c}}[\Theta]\}
$, $P_*^\perp$ to be the projection on $\sigma_d(\mathcal{H}_{\omega_{*}})$ given by Riesz projection and $P_*:=1-P_*^\perp$.

%

\begin{lemma}\label{lem:R}
There exists $\delta>0$ s.t.\ for $(\varpi,\mathbf{v},\mathbf{z})$ satisfying $|\varpi|+\|\mathbf{v}\|+\|\mathbf{z}\|<\delta$, the projection $P_*$ restricted on $\widetilde{\mathcal{H}}_{\mathrm{c}}[0,\varpi,0,\mathbf{v},\mathbf{z}]$ is invertible.
Furthermore, setting $$R[\varpi,\mathbf{v},\mathbf{z}]:=(\left.P_*\right|_{\widetilde{\mathcal{H}}_{\mathrm{c}}[0,\varpi,0,\mathbf{v},\mathbf{z}]})^{-1},$$
 we have
\begin{align}\label{est:Rminus1}
\|R-1\|_{\Sigma^{*}\to \Sigma}\lesssim |\varpi|+\|\mathbf{v}\|+\|\mathbf{z}\|,\quad
\|\partial_x R\|_{\Sigma^{*}\to \Sigma}\lesssim 1,
\end{align}
for $\ x=\omega,v_l,z_{jA}$, $l=1,2,3$, $j=1,\cdots,N$ and $A=R,I.$
\end{lemma}

\begin{proof}
The proof is standard.
\end{proof}

We define $\widetilde{G}_{\theta,\mathbf{y},\mathbf{v}}$ by $\widetilde{G}_{\theta,\mathbf{y},\mathbf{v}}w=e^{\im \sigma_3\theta} e^{\im \sigma_3 \frac{1}{2}\mathbf{v}\cdot (x-\mathbf{y})}u(t,x-\mathbf{y})$
and consider
\begin{align}\label{eq:eta_tildeeta_1}
\widetilde{\zeta}:=P_* \widetilde{G}_{-\theta,-\mathbf{y},0}\zeta.
\end{align}
By Lemma \ref{lem:R}, we have
\begin{align}\label{eq:eta_tildeeta_2}
\zeta=\widetilde{G}_{\theta,\mathbf{y},0}R[\varpi,\mathbf{v},\mathbf{z}]\widetilde{\zeta}.
\end{align}
One can also check $\mathcal{H}[\widetilde{\Theta}]=G_{\theta,\mathbf{y},0}\mathcal{H}[0,\varpi,0,\mathbf{v},\mathbf{z}]G_{-\theta,-\mathbf{y},0}$ using \eqref{def:rp}.
From \eqref{eq:eta_tildeeta_1}, \eqref{eq:eta_tildeeta_2} and Lemma \ref{lem:R}, we have
\begin{align}\label{eq:eta_tildeeta_3}
\|\eta\|_{\mathrm{Stz}}\sim\|\zeta\|_{\mathrm{Stz}}\sim \|\widetilde{\zeta}\|_{\mathrm{Stz}}.
\end{align}
Substituting \eqref{eq:eta_tildeeta_2} into \eqref{eq:modnlsvec} and applying $ P_* \widetilde{G}_{-\theta,-\mathbf{y},0}$, we have
\begin{align}\label{eq:tildeeta}
\im \partial_t \widetilde{\zeta} = \mathcal{H}_{\omega_*}  \widetilde{\zeta} + P_*\( \sigma_3\dot{\theta} \widetilde{\zeta}+\im \sum_{l=1}^{3}\dot{y_l}\partial_{x_l}\widetilde{\zeta}\)
+P_* \widetilde{G}_{-\theta,-\mathbf{y},0}\sigma_3\mathfrak{F}
 + P_*\sigma_3 \mathfrak{R}[0,\varpi,0,\mathbf{v},\mathbf{z}]  + R_{\widetilde{\zeta}},
\end{align}
where
\begin{align*}
R_{\widetilde{\zeta}}&=-\im P_* \(\im \sigma_3 \dot{\theta}(R-1)\widetilde{\zeta} -\sum_{l=1}^3\dot{y_l}\partial_{x_l}\((R-1)\zeta\) + \dot{\varpi}(\partial_{\varpi}R)\widetilde{\zeta}+ \(D_{\mathbf{v}}R \dot{\mathbf{v}}\)\widetilde{\zeta}+\(D_{\mathbf{z}}R \dot{\mathbf{z}}\)\widetilde{\zeta}\)\\&
\quad - \im P_*
D\phi[0,\varpi,0,\mathbf{v},\mathbf{z}] (\dot{\Theta}-\widetilde{\Theta})+P_*\(\mathcal{H}[0,\varpi,0,\mathbf{v},\mathbf{z}]-\mathcal{H}_{\omega_*}\)R\widetilde{\zeta}.
\end{align*}

We now recall Beceanu's Strichartz estimate \cite{beceanu}.
We denote by $\mathcal{U}(t,s)$ the evolution operator associated to the equation
\begin{align*}
\im \partial_t w= \mathcal{H}_{\omega_{*}}w + P_{*} \( \sigma_3  a_0(t)w + \im \sum_{l=1}^{3}a_l(t)\partial_{x_l} w\),\ w=P_*w.
\end{align*}

\begin{proposition}\label{prop:Stz}
There exists $\delta>0$ s.t.\ if $\|a_j(t)\|_{L^\infty}<\delta$, then for $u_0=P_*u_0$ and $\mathcal{F}=P_*\mathcal{F}$, we have
\begin{align*}
\|\mathcal{U}(t,0)u_0\|_{\mathrm{Stz} }&\lesssim \|u(0)\|_{H^1},\\
\|\int_0^\cdot \mathcal{U}(t,s)\mathcal{F}(s)\,ds\|_{\mathrm{Stz} }&\lesssim \|\mathcal{F}\|_{L^2W^{1,6/5}}.
\end{align*}
\end{proposition}

\begin{proof}
See, \cite{beceanu}.
\end{proof}

Using  Strichartz estimates and  Proposition \ref{prop:Stz}, we now estimate   $\|\eta\|_{\mathrm{Stz} }$.

\begin{lemma}\label{lem:etabound}
Under the assumption of Proposition \ref{prop:boot}, we have
\begin{align}\label{etabound}
\|\eta\|_{\mathrm{Stz} }\lesssim \epsilon + \sum_{\mathbf{m}\in \mathbf{R}_{\mathrm{min}}} \|\mathbf{z}^{\mathbf{m}}\|_{L^2}.
\end{align}
\end{lemma}

\begin{proof}
From \eqref{eq:Thetadot}, \eqref{eq:eta_tildeeta_3}, \eqref{eq:tildeeta} and  Proposition \ref{prop:Stz} we have
\begin{align*}
\|\eta\|_{\mathrm{Stz} }\lesssim \epsilon + \|\mathfrak{F}\|_{L^2W^{1,6/5}}+\|\mathfrak{R}[0,\varpi,0,\mathbf{v},\mathbf{z}]\|_{L^2W^{1,6/5}}+\|R_{\widetilde{\eta}}\|_{L^2W^{1,6/5}}.
\end{align*}
Since $\|\mathfrak{F}\|_{L^{2}W^{1,6/5}}\sim \|F\|_{L^2 W^{1,6/5}}$, by Lemma \ref{lem:Fbound}, we have
\begin{align*}
\|\mathfrak{F}\|_{L^2 W^{1,6/5}}\lesssim \epsilon \|\eta\|_{\mathrm{Stz} }.
\end{align*}
Using \eqref{Theta:L2}, \eqref{eq:Thetadot} and Lemma \ref{lem:R} we have
\begin{align*}
\|R_{\widetilde{\eta}}\|_{L^2 W^{1,6/5}}\lesssim \epsilon \|\eta\|_{\mathrm{Stz^1}}.
\end{align*}
Thus, we obtain \eqref{etabound}.
\end{proof}

In the rest of the paper we focus on the bound on $\sum_{\mathbf{m}\in \mathbf{R}_{\mathrm{min}}} \|\mathbf{z}^{\mathbf{m}}\|_{L^2}$. To this effect, for $\mathfrak{G}_{\mathbf{m}}$ is given in \eqref{FermiG},  we need  the expansion
\begin{align}\label{def:g} &
g=\widetilde{\zeta}-Z \text{ with }Z =\sum_{\mathbf{m} \in \mathbf{R}_{\mathrm{min}}}\mathbf{z}^{\mathbf{m}}Z_{\mathbf{m}} \text{  with}\\& Z_{\mathbf{m}}
=
-(\mathcal{H}_{\omega_*}-\lambda(\omega_*,\mathbf{m})-\im 0)^{-1}\sigma_3\mathfrak{G}_{\mathbf{m}}. \nonumber
\end{align}
Then,  for $\mathfrak{R}_1=\mathfrak{R}[0,\varpi,0,\mathbf{v},\mathbf{z}]-\mathfrak{G}_{\mathbf{m}}$, the vector valued function  $g$ solves
\begin{align}\label{eq:g}
\im \partial_t g = &\mathcal{H}_{\omega_*}  g+ P_*\( \sigma_3\dot{\theta} g+\im \sum_{l=1}^{3}\dot{y_l}\partial_{x_l}g\) + P_*\(\sigma_3 \dot{\theta} Z+\im \sum_{l=1}^{3}\dot{y_l}\partial_{x_l}Z\)\\&
+P_* \widetilde{G}_{-\theta,-\mathbf{y},0}\sigma_3\mathfrak{F}
 + P_*\sigma_3 \mathfrak{R}_1  + R_{\widetilde{\zeta}}
-\sum_{\mathbf{m}\in \mathbf{R}_{\mathrm{min}}}\(\im \partial_t(\mathbf{z}^{\mathbf{m}})-\lambda(\omega_*,\mathbf{m})\mathbf{z}^{\mathbf{m}}\)Z_{\mathbf{m}}.\nonumber
\end{align}
Before going into the estimates of $g$, we    consider several  preparatory  lemmas.
\begin{lemma}\label{lem:zmest}
Let $\mathbf{m}\in \mathbf{R}_{\mathrm{min}}$.
We have
\begin{align*}
\|\im \partial_t(\mathbf{z}^{\mathbf{m}})-\lambda(\omega_*,\mathbf{m})\mathbf{z}^{\mathbf{m}}\|_{L^2}\lesssim \epsilon (C_0\epsilon).
\end{align*}
\end{lemma}

\begin{proof}
Setting $\boldsymbol{\lambda}\mathbf{z}:=(\lambda_1(\omega_*)z_1,\cdots,\lambda_N(\omega_*)z_N)$, we have
\begin{align}\label{estgpre1}
\partial_t(\mathbf{z}^{\mathbf{m}})+\im \lambda(\omega_*,\mathbf{m})\mathbf{z}^{\mathbf{m}}=D_{\mathbf{z}}(\mathbf{z}^{\mathbf{m}})\(\dot{\mathbf{z}}-\widetilde{\mathbf{z}}\)+D_{\mathbf{z}}(\mathbf{z}^{\mathbf{m}})\(\widetilde{\mathbf{z}}+\im \boldsymbol{\lambda}\mathbf{z}\).
\end{align}
Since $\mathbf{m}\in \mathbf{R}_{\mathrm{min}}$ implies $\|\mathbf{m}\|\geq 2$, from \eqref{Theta:L2}, we have
\begin{align}\label{estgpre2}
\|D_{\mathbf{z}}(\mathbf{z}^{\mathbf{m}})\(\dot{\mathbf{z}}-\widetilde{\mathbf{z}}\)\|_{L^2}\lesssim \epsilon \(\epsilon\|\eta\|_{\mathrm{Stz} }+\|\eta\|_{\mathrm{Stz} }^2\).
\end{align}
Next, from \eqref{eq:zanz}, we have
\begin{align}\label{estgpre3}
D_{\mathbf{z}}(\mathbf{z}^{\mathbf{m}})\(\widetilde{\mathbf{z}}+\im \boldsymbol{\lambda}\mathbf{z}\)=&\sum_{j=1}^{N} m_{j+}\mathbf{z}^{\mathbf{m}-\mathbf{e}^{j+}}\(\im \sum_{\mathbf{n}\in \Lambda_j,\ \|\mathbf{n}\|\geq 2}\mathbf{z}^{\mathbf{n}}\widetilde{\lambda}_{j\mathbf{m}}(\omega)+\widetilde{z}_{j\mathcal{R}}(\omega,\mathbf{z})\)\\&+\sum_{j=1}^{N}m_{j-}\mathbf{z}^{\mathbf{m}-\mathbf{e}^{j-}}\overline{\(\im \sum_{\mathbf{n}\in \Lambda_j,\ \|\mathbf{n}\|\geq 2}\mathbf{z}^{\mathbf{n}}\widetilde{\lambda}_{j\mathbf{n}}(\omega)+\widetilde{z}_{j\mathcal{R}}(\omega,\mathbf{z})\)}.\nonumber
\end{align}
By \eqref{eq:Ranz}, we have
\begin{align}\label{estgpre4}
\sum_{j=1}^{N}\|  m_{j+}\mathbf{z}^{\mathbf{m}-\mathbf{e}^{j+}}\widetilde{z}_{j\mathcal{R}}(\omega,\mathbf{z})+m_{j-}\mathbf{z}^{\mathbf{m}-\mathbf{e}^{j-}}\overline{\widetilde{z}_{j\mathcal{R}}(\omega,\mathbf{z})}\|_{L^2}\lesssim \epsilon \sum_{\mathbf{m} \in \mathbf{R}_{\mathrm{min}}}\|\mathbf{z}^{\mathbf{m}}\|_{L^2}
\end{align}
We are left to estimate terms in the form $m_{j+}\mathbf{z}^{\mathbf{m}-\mathbf{e}^{j+}+\mathbf{n}}$ or $m_{j-}\mathbf{z}^{\mathbf{m}-\mathbf{e}^{j-}+\overline{\mathbf{n}}}$ with $\mathbf{n}\in \Lambda_j$ and $\|\mathbf{n}\|\geq 2$.
We will only handle the former and the latter can be estimated by similar argument.
Since $\mathbf{m}\in \mathbf{R}_{\mathrm{min}}$, if $m_{j+}\neq 0$, we have $\lambda(\omega_*,\mathbf{m})>\omega_*$ and $\mathbf{m}_-=0$.
From assumption (H6), we have $\mathbf{n}_-\geq  1$.
Since
\begin{align*}
\omega_*<\lambda(\omega_*,\mathbf{m})=\sum_{k=1}^{N}\lambda_k(m_{k+}+\delta_{jk}+n_{k+}-n_{k-})<\sum_{k=1}^{N}\lambda_k(m_{k+}+\delta_{jk}+n_{k+}),
\end{align*}
we have $\mathbf{m}'  \preceq \mathbf{m}-\mathbf{e}^{j+}+(\mathbf{n}_+,0)$ for some $\mathbf{m}'\in \mathbf{R}_{\mathrm{min}}$.
Thus, $\mathbf{m}-\mathbf{e}^{j+}+\mathbf{n}\in \mathbf{I}$.
Therefore,
\begin{align}\label{estgpre5}
\sum_{j=1}^{N} \|m_{j+}\mathbf{z}^{\mathbf{m}-\mathbf{e}^{j+}} \sum_{\mathbf{n}\in \Lambda_j,\ \|\mathbf{n}\|\geq 2}\mathbf{z}^{\mathbf{n}}\widetilde{\lambda}_{j\mathbf{m}}(\omega)+m_{j-}\mathbf{z}^{\mathbf{m}-\mathbf{e}^{j-}}\overline{\sum_{\mathbf{n}\in \Lambda_j,\ \|\mathbf{n}\|\geq 2}\mathbf{z}^{\mathbf{n}}\widetilde{\lambda}_{j\mathbf{n}}(\omega)}\|_{L^2}\lesssim \epsilon \sum_{\mathbf{m} \in \mathbf{R}_{\mathrm{min}}}\|\mathbf{z}^{\mathbf{m}}\|_{L^2}.
\end{align}
Form \eqref{estgpre1}, \eqref{estgpre2}, \eqref{estgpre3}, \eqref{estgpre4} and \eqref{estgpre5}, we have the conclusion.
\end{proof}

\begin{lemma}\label{prop:outgoing}
Let $\sigma> 0$ sufficiently large and $|\lambda|>\omega_*$.
Then, there exists $\delta>0$ s.t.\ if $\|a_j(t)\|_{L^\infty}<\delta$,
\begin{align*}
\|\mathcal{U}(t,0)(\mathcal{H}_{\omega_*}-\lambda - \im 0)^{-1}G\|_{L^{2,-\sigma}}&\lesssim  \<t\>^{-3/2}\|G\|_{L^{2,\sigma}},\ t\geq 0,\\
\|\int_0^\cdot \mathcal{U}(t,s)(\mathcal{H}_{\omega_*}-\lambda - \im 0)^{-1}\mathcal{F}\,ds\|_{L^2L^{2,-\sigma}}&\lesssim \|\mathcal{F}\|_{L^2L^{2,\sigma}}.
\end{align*}
\end{lemma}

\begin{proof}
See, e.g.\ Lemma 10.1 and 10.2 of \cite{CM14}.
\end{proof}

We can now estimate $g$.
\begin{lemma}\label{lem:g}
We have
\begin{align*}
\|g\|_{L^2L^{2,-\sigma}} \lesssim \epsilon + \epsilon(C_0\epsilon).
\end{align*}
\end{lemma}

\begin{proof}
From \eqref{eq:g}, we have
\begin{align}
\|g\|_{L^2L^{2,-\sigma}}\lesssim& \|\mathcal{U}(t,0)\widetilde{\zeta}(0)\|_{\mathrm{Stz^1}}+\sum_{\mathbf{m} \in \mathbf{R}_{\mathrm{min}}}|\mathbf{z}(0)^{\mathbf{m}}|\|\mathcal{U}(t,0)Z_{\mathbf{m}}\|_{L^2L^{2,-\sigma}}+
\|\mathfrak{F}\|_{\mathrm{Stz} }+\|\mathfrak{R}_1\|_{\mathrm{Stz} }+\|R_{\widetilde{\zeta}}\|_{\mathrm{Stz} }\nonumber\\&+\sum_{\mathbf{m} \in \mathbf{R}_{\mathrm{min}}} \|\int_0^\cdot \mathcal{U}(\cdot,s)\mathbf{z}(s)^{\mathbf{m}}P_*\(\sigma_3 \dot{\theta} Z_{\mathbf{m}}+\im \sum_{l=1}^{3}\dot{y_l}\partial_{x_l}Z_{\mathbf{m}}\)\|_{L^2L^{2,-\sigma}}\nonumber\\&
+\sum_{\mathbf{m} \in \mathbf{R}_{\mathrm{min}}}\|\int_0^\cdot \mathcal{U}(t,s)\(\im \partial_t(\mathbf{z}^{\mathbf{m}}(s))-\lambda(\omega_*,\mathbf{m})\mathbf{z}^{\mathbf{m}}(s)\)Z_{\mathbf{m}}\|_{L^2L^{2,-\sigma}},\label{gDuha}
\end{align}
where we have used $\mathrm{Stz} \hookrightarrow L^2L^{2,-\sigma}$.
The terms with $\mathrm{Stz} $ norms can be estimated similarly as the proof of Lemma \ref{lem:etabound} with the bound $\epsilon(C_0\epsilon)$.
The remaining terms can be bounded using Lemmas \ref{lem:zmest},  \ref{prop:outgoing}.
Thus, we have the conclusion.
\end{proof}

\subsection{Fermi Golden Rule estimate}

We are finally ready to estimate $\sum_{\mathbf{m} \in \mathbf{R}_{\mathrm{min}}}\|\mathbf{z}^{\mathbf{m}}\|_{L^2}$.

\begin{lemma}\label{lem:FGRbound}
We have
\begin{align*}
\sum_{\mathbf{m} \in \mathbf{R}_{\mathrm{min}}}\|\mathbf{z}\|_{L^2} \leq \frac{1}{2}C_0\epsilon.
\end{align*}
\end{lemma}

\begin{proof}
We start by computing the time derivative of the localized action $E(\varphi[\Theta])+\omega_*Q_0(\varphi[\Theta])$.
\begin{align}
\frac{d}{dt}\(E(\varphi)+\omega_{*}Q_0(\varphi)\) &= \< (-\Delta+\omega_*)\varphi +g(|\varphi|^2)\varphi,D\varphi \dot{\Theta}\>=\<\im D\varphi \widetilde{\Theta},D\varphi  \dot{\Theta}\>\nonumber\\&
=-\<\im D\varphi (\dot{\Theta}-\widetilde{\Theta}),D\varphi \widetilde{\Theta}\>\nonumber\\&
=-\<H\eta +F,D\varphi \widetilde{\Theta}\>+\<\im \partial_t \eta,D\varphi \widetilde{\Theta}\>\nonumber\\&
=-\<\eta,\im D^2\varphi(\widetilde{\Theta},\widetilde{\Theta})+D\mathcal{R}\widetilde{\Theta}\>-\<F,D\varphi\widetilde{\Theta}\>
+\<\eta,\im D^2\varphi(\dot{\Theta},\widetilde{\Theta})\>\nonumber\\&
=-\<\eta,D\mathcal{R}\widetilde{\Theta}\>-\<\eta,\im D^2\varphi (\dot{\Theta}-\widetilde{\Theta},\widetilde{\Theta})\>-\<F,D\varphi\widetilde{\Theta}\>,\label{eq:LocEn1}
\end{align}
where we have used \eqref{prop:full_rp} in the 1st line, $\<\im w,w\>=0$ in the 2nd line, \eqref{eq:modest} in the 3rd line, \eqref{eq:HDphi} in the 4th line.
From \eqref{F:L1est} and \eqref{Theta:L2} we have
\begin{align}\label{eq:LocEn2}
\|\<\eta,\im D^2\varphi (\dot{\Theta}-\widetilde{\Theta},\widetilde{\Theta})\>\|_{L^1}+\|\<F,D\varphi\widetilde{\Theta}\>\|_{L^1}\lesssim  \epsilon\(C_0\epsilon\)^2.
\end{align}
For the 1st term of the last line of \eqref{eq:LocEn1},
\begin{align}\label{eq:LocEn3}
-\<\eta,D\mathcal{R}\widetilde{\Theta}\>=\<\eta,D_{\mathbf{z}}\mathcal{R}(\im \boldsymbol{\lambda}\mathbf{z})\>-\<\eta,D\mathcal{R}\(\widetilde{\Theta}-(0,0,0,0,-\im \boldsymbol{\lambda}\mathbf{z})\)\>
\end{align}
where $\boldsymbol{\lambda}\mathbf{z}=(\lambda_1(\omega_*)z_1,\cdots,\lambda_N(\omega_{*})z_N)$.
Following the argument of Lemma \ref{lem:zmest}, we have
\begin{align}\label{eq:LocEn4}
\|\<\eta,D\mathcal{R}\(\widetilde{\Theta}-(0,0,0,0,-\im \boldsymbol{\lambda}\mathbf{z})\)\>\|_{L^1}\lesssim C_0^2 \epsilon^3.
\end{align}
For the 1st term in the r.h.s.\ of \eqref{eq:LocEn3},
\begin{align}\label{eq:LocEn5}
\<\eta,D_{\mathbf{z}}\mathcal{R}\( \im \boldsymbol{\lambda}\mathbf{z}\)\>=&
\<G_{-\theta,-\mathbf{y},0}\eta,\im \sum_{\mathbf{m} \in \mathbf{R}_{\mathrm{min}}}\lambda(\omega_*,\mathbf{m})\mathbf{z}^{\mathbf{m}}G_{\mathbf{m}}\>\\&+\<G_{-\theta,-\mathbf{y},0}\eta,,\im \sum_{\mathbf{m} \in \mathbf{R}_{\mathrm{min}}}\lambda(\omega_*,\mathbf{m})\mathbf{z}^{\mathbf{m}}\(G_{0,0,\mathbf{v}}G_{\mathbf{m}}-G_{\mathbf{m}}\)\>+\<\eta,D_{\mathbf{z}}\mathcal{R}_1\( \im \boldsymbol{\lambda}\mathbf{z}\)\>.\nonumber
\end{align}
By \eqref{estR} and \eqref{Linftybound}, the terms in the 2nd line of \eqref{eq:LocEn5} can be bounded by
\begin{align}\label{eq:LocEn6}
\|\<G_{-\theta,-\mathbf{y},0}\eta,,\im \sum_{\mathbf{m} \in \mathbf{R}_{\mathrm{min}}}\lambda(\omega_*,\mathbf{m})\mathbf{z}^{\mathbf{m}}\(G_{0,0,\mathbf{v}}G_{\mathbf{m}}-G_{\mathbf{m}}\)\>+\<\eta,D_{\mathbf{z}}\mathcal{R}_1\( \im \boldsymbol{\lambda}\mathbf{z}\)\>\|_{L^1}\lesssim \epsilon (C_0\epsilon)^2.
 \end{align}
 For the 1st term in the r.h.s.\ of \eqref{eq:LocEn5}, since $\mathbf{R}_{\mathrm{min}}=\{\mathbf{m},\overline{\mathbf{m}}\ |\ \lambda(\omega_{*},\mathbf{m})>\omega_*\}$, setting
 \begin{align*}
 \mathbf{R}_{\mathrm{min}+}:=\{\mathbf{m}\in \mathbf{R}_{\mathrm{min}},\ \lambda(\omega_{*},\mathbf{m})>\omega_{*}\},
 \end{align*}
  we have
 \begin{align}\nonumber
& \<G_{-\theta,-\mathbf{y},0}\eta,\im \sum_{\mathbf{m} \in \mathbf{R}_{\mathrm{min}}}\lambda(\omega_*,\mathbf{m})\mathbf{z}^{\mathbf{m}}G_{\mathbf{m}}\>\\&
=\<G_{-\theta,-\mathbf{y},0}\eta,\sum_{\substack{\mathbf{m}\in \mathbf{R}_{\mathrm{min}+}}}\( \im \lambda(\omega_*,\mathbf{m})\mathbf{z}^{\mathbf{m}}G_{\mathbf{m}}+\im \lambda(\omega_*,\overline{\mathbf{m}})\mathbf{z}^{\overline{\mathbf{m}}}G_{\overline{\mathbf{m}}}\)\>\nonumber\\&
 =\sum_{\substack{\mathbf{m}\in \mathbf{R}_{\mathrm{min}+}}}\(\<G_{-\theta,-\mathbf{y},0}\eta, \im \lambda(\omega_*,\mathbf{m})\mathbf{z}^{\mathbf{m}}G_{\mathbf{m}}\>+\<\overline{G_{-\theta,-\mathbf{y},0}\eta},\im \lambda(\omega_*,\mathbf{m})\mathbf{z}^{\mathbf{m}}G_{\overline{\mathbf{m}}}\>\)\nonumber\\&
 =\<\widetilde{G}_{-\theta,-\mathbf{y},0}\zeta,\sum_{\substack{\mathbf{m}\in \mathbf{R}_{\mathrm{min}+}
 }}\im \lambda(\omega_*,\mathbf{m})\mathbf{z}^{\mathbf{m}}\mathfrak{G}_{\mathbf{m}}\>.\label{eq:LocEn61}
 \end{align}
By \eqref{def:g}, the 1st term in the r.h.s.\ of \eqref{eq:LocEn61} can be decomposed as
\begin{align}\label{eq:LocEn7}
&\<\widetilde{G}_{-\theta,-\mathbf{y},0}\zeta,\sum_{\substack{\mathbf{m}\in \mathbf{R}_{\mathrm{min}+}
}}\im \lambda(\omega_*,\mathbf{m})\mathbf{z}^{\mathbf{m}}\mathfrak{G}_{\mathbf{m}}\>=
\<Z+g+(R-1)\widetilde{\zeta},\sum_{\mathbf{m}\in \mathbf{R}_{\mathrm{min}+}}\im \lambda(\omega_*,\mathbf{m})\mathbf{z}^{\mathbf{m}}\mathfrak{G}_{\mathbf{m}}\> .
\end{align}
By Lemmas \ref{lem:R} and \ref{lem:g} we have
\begin{align}\label{eq:LocEn8}
\|\<g+(R-1)\widetilde{\zeta},\sum_{\mathbf{m}\in \mathbf{R}_{\mathrm{min}+}}\im \lambda(\omega_*,\mathbf{m})\mathbf{z}^{\mathbf{m}}\mathfrak{G}_{\mathbf{m}}\>\|_{L^1}\lesssim \epsilon (1+C_0\epsilon)C_0\epsilon.
\end{align}
Recalling \eqref{def:Rmink}, the remaining term of r.h.s.\ of  \eqref{eq:LocEn7} can be decomposed as
\begin{align}\label{eq:LocEn9}
\<Z,\sum_{\mathbf{m}\in \mathbf{R}_{\mathrm{min}+}}\im \lambda(\omega_*,\mathbf{m})\mathbf{z}^{\mathbf{m}}\mathfrak{G}_{\mathbf{m}}\> =&\sum_{k=1}^M \sum_{\mathbf{m}^1,\mathbf{m}^2\in \mathbf{R}_{\mathrm{min},k}}r_k\<\mathbf{z}^{\mathbf{m}^1}Z_{\mathbf{m}^1},\im \mathbf{z}^{\mathbf{m}^2}\mathfrak{G}_{\mathbf{m}^2}\>
\\&+\sum_{\substack{\mathbf{m}^{1}\in \mathbf{R}_{\mathrm{min}},\mathbf{m}^{2} \in \mathbf{R}_{\mathrm{min}+}\\ \lambda(\omega_*,\mathbf{m}^1)\neq \lambda(\omega_*,\mathbf{m}^2)}}
\<\mathbf{z}^{\mathbf{m}^1}Z_{\mathbf{m}^1},\im \lambda(\omega_*,\mathbf{m}^2)\mathbf{z}^{\mathbf{m}^2}\mathfrak{G}_{\mathbf{m}^2}\>.\nonumber
\end{align}
Now, for $\mathbf{m}^1$ and $\mathbf{m}^2$ with $\lambda(\omega_*,\mathbf{m}^1)\neq \lambda(\omega_*,\mathbf{m}^2)$, we have
\begin{align*}
\frac{d}{dt}\mathbf{z}^{\mathbf{m}^1+\overline{\mathbf{m}^2}}=\im (\lambda(\omega_*,\mathbf{m}^1)-\lambda(\omega_*,\mathbf{m}^2))\(\mathbf{z}^{\mathbf{m}^1+\overline{\mathbf{m}^2}}+r_{\mathbf{m}_1,\mathbf{m}_2}\),
\end{align*}
where
\begin{align*}
r_{\mathbf{m}_1,\mathbf{m}_2}=\frac{1}{\im (\lambda(\omega_*,\mathbf{m}^1)-\lambda(\omega_*,\mathbf{m}^2))}D(\mathbf{z}^{\mathbf{m}^1+\overline{\mathbf{m}^2}})(\dot{\mathbf{z}}+\im \boldsymbol{\lambda}\mathbf{z}).
\end{align*}
Arguing as Lemma \ref{lem:zmest}, we can show
\begin{align}\label{eq:LocEn10}
\|r_{\mathbf{m}^1,\mathbf{m}^2}\|_{L^1}\lesssim \epsilon(C_0\epsilon)^2.
\end{align}
Thus, for the 2nd term in the r.h.s.\ of \eqref{eq:LocEn9},
\begin{align}
\<\mathbf{z}^{\mathbf{m}^1}Z_{\mathbf{m}^1},\im \lambda(\omega_*,\mathbf{m}^2)\mathbf{z}^{\mathbf{m}^2}\mathfrak{G}_{\mathbf{m}^2}\>
=\frac{d}{dt}A
_{\mathbf{m}^1,\mathbf{m}^2}
-
r_{\mathbf{m}^1,\mathbf{m}^2}\<Z_{\mathbf{m}^1},\im \lambda(\omega_*,\mathbf{m}^2)\mathfrak{G}_{\mathbf{m}^2}\>,\label{eq:LonEn11}
\end{align}
where
\begin{align}\label{eq:LocEn12}
A_{\mathbf{m}_1,\mathbf{m}_2}=
\frac{1}{\im (\lambda(\omega_*,\mathbf{m}^1)-\lambda(\omega_*,\mathbf{m}^2))}\<\mathbf{z}^{\mathbf{m}^1}Z_{\mathbf{m}^1},\im \lambda(\omega_*,\mathbf{m}^2)\mathbf{z}^{\mathbf{m}^2}\mathfrak{G}_{\mathbf{m}^2}[\omega_{*}]\>.
\end{align}
Finally, for the 1st term of r.h.s.\ of \eqref{eq:LocEn9}, by Plemelj formula  we have
\begin{align}\label{eq:LocEn13}
&\sum_{k=1}^M \sum_{\mathbf{m}^1,\mathbf{m}^2\in \mathbf{R}_{\mathrm{min},k}}r_k\<\mathbf{z}^{\mathbf{m}^1}Z_{\mathbf{m}^1},\im \mathbf{z}^{\mathbf{m}^2}\mathfrak{G}_{\mathbf{m}^2}[\omega_{*}]\>\\&=-\sum_{k=1}^{M}\frac{\pi}{2\sqrt{r_k-\omega_*}}\int_{|\xi|^2=r_k-\omega_*}|\sum_{\mathbf{m}\in \mathbf{R}_{\mathrm{min}},k}\mathbf{z}^{\mathbf{m}}\mathcal{F}{\(W^*\mathfrak{G}_{\mathbf{m}}\)_+}|^2\,d\sigma(\xi). \nonumber
\end{align}
Here, $W^*$ is the adjoint of the wave operator $W$ given by \eqref{def:waveop}, $\mathcal{F}$ is the usual Fourier transform and $F_+=f_1$ for $F={}^t(f_1\ f_2)$.
Since, for each $k$, the r.h.s.\ of \eqref{eq:LocEn13} is a non-negative bilinear form of $\mathbf{z}^{\mathbf{m}}$'s, by (H7), we have
\begin{align}\label{eq:LocEn14}
\sum_{k=1}^{N}\int_{|\xi|^2=r_k}|\sum_{\mathbf{m}\in \mathbf{R}_{\mathrm{min}},k}\mathbf{z}^{\mathbf{m}}\mathcal{F}{\(W^*\mathfrak{G}_{\mathbf{m}}\)_+}|^2\,d\sigma\gtrsim \sum_{\mathbf{m} \in \mathbf{R}_{\mathrm{min}}} |\mathbf{z}^{\mathbf{m}}|^2
\end{align}
Collecting all \eqref{eq:LocEn1}-\eqref{eq:LocEn14}, we have
\begin{align}\label{eq:LocEn15}
\sum_{\mathbf{m} \in \mathbf{R}_{\mathrm{min}}}\|\mathbf{z}^\mathbf{m}\|_{L^2}^2\lesssim \left[E(\varphi)+\omega_* Q_0 + A\right]_0^T + \epsilon(C_0\epsilon)^2 + \epsilon(C_0\epsilon).
\end{align}
By Orbital stability, Proposition \ref{prop:os}, we see that the 1st term in r.h.s.\ of \eqref{eq:LocEn15} can be bounded by $\lesssim \epsilon^2$.
So, we have
\begin{align}\label{eq:LocEn16}
\sum_{\mathbf{m} \in \mathbf{R}_{\mathrm{min}}}\|\mathbf{z}^\mathbf{m}\|_{L^2}^2\leq C(\epsilon^2+\epsilon (C_0\epsilon)^2+\epsilon (C_0\epsilon))
\end{align}
Thus, taking $C_0$ sufficiently large so that $C(1+C_0)\leq \frac{1}{100N}C_0^2$ and $\epsilon_0$ sufficiently small so that $C\epsilon_0\leq\frac{1}{100N}$, we have
\begin{align}\label{eq:LocEn16}
\sum_{\mathbf{m} \in \mathbf{R}_{\mathrm{min}}}\|\mathbf{z}^\mathbf{m}\|_{L^2}^2\leq \frac{1}{10N}(C_0\epsilon)^2.
\end{align}
Therefore, we have the conclusion.
\end{proof}

\begin{proof}[Proof of Proposition \ref{prop:boot}]
Taking $C_0$ larger if necessary, from Lemmas \ref{lem:etabound} and \ref{lem:FGRbound}, we have \eqref{eq:boot} with $C_0$ replaced by $C_0/2$.
\end{proof}

\section*{Acknowledgments}
C. was supported by a FRA of the University of Trieste and by the Prin 2020 project \textit{Hamiltonian and Dispersive PDEs} n. 2020XB3EFL.
M.M. was supported by the JSPS KAKENHI Grant Number 19K03579, G19KK0066A and JP17H02853.

Department of Mathematics and Geosciences,  University
of Trieste, via Valerio  12/1  Trieste, 34127  Italy.
{\it E-mail Address}: {\tt scuccagna@units.it}

Department of Mathematics and Informatics,
Graduate School of Science,
Chiba University,
Chiba 263-8522, Japan.
{\it E-mail Address}: {\tt maeda@math.s.chiba-u.ac.jp}

\end{document}